\documentclass[a4paper,12pt,draft]{amsart}

\usepackage[margin=30mm]{geometry}
\usepackage{amssymb, amsmath, amsthm, amscd}

\usepackage{eucal,mathrsfs,dsfont}
\usepackage{color}

\renewcommand{\leq}{\leqslant}
\renewcommand{\geq}{\geqslant}
\renewcommand{\le}{\leqslant}
\renewcommand{\ge}{\geqslant}

\definecolor{mno}{rgb}{0.5,0.1,0.5}

\newcommand{\R}{\mathbb R}
\newcommand{\Rd}{{\mathbb R^d}}
\newcommand{\Pp}{{\bf P}}
\newcommand{\Ee}{{\bf E}}

\newcommand{\I}{\mathds 1}

\newcommand{\eps}{\varepsilon}
\newcommand{\dlt}{{(\delta)}}
\newcommand{\dltp}{{(\delta')}}

\newtheorem{theorem}{Theorem}[section]
\newtheorem{lemma}[theorem]{Lemma}
\newtheorem{proposition}[theorem]{Proposition}

\numberwithin{equation}{section}
\newtheorem{question}{Question}[section]

\theoremstyle{definition}

\newtheorem{example}[theorem]{Example}
\newtheorem{remark}[theorem]{Remark}
\newtheorem{assumption}[theorem]{Assumption}

\begin{document}
\allowdisplaybreaks
\title[LILs for symmetric jump processes] {\bfseries
Laws of the iterated logarithm for symmetric jump processes}
\author{Panki Kim \quad Takashi Kumagai
\quad Jian Wang}
\thanks{The research of Panki Kim is supported by
 the National Research Foundation of Korea(NRF) grant funded by the Korea government(MSIP)
(No. NRF-2015R1A4A1041675). \
 The research of Takashi Kumagai is partially supported
by the Grant-in-Aid for Scientific Research (A) 25247007, Japan.\ The research of Jian Wang is supported by National
Natural Science Foundation of China (No.\ 11201073 and 11522106), the JSPS postdoctoral fellowship
(26$\cdot$04021), and the Program for Nonlinear Analysis and Its Applications (No. IRTL1206)}

\date{}
\maketitle

\begin{abstract} Based on two-sided heat kernel estimates for a class of symmetric jump processes on metric measure spaces, the laws of the iterated logarithm (LILs) for sample paths, local times and ranges are established. In particular, the LILs are obtained for $\beta$-stable-like processes on $\alpha$-sets with $\beta>0$.


\medskip

\end{abstract}
\allowdisplaybreaks

\section{Introduction and Setting}\label{section1}

The law of the iterated logarithm (LIL) describes the magnitude of the fluctuations of stochastic processes. The original statement of LIL for a random walk is due to  Khinchin in \cite{Kh}.
In this paper we discuss various types of the LILs for a large class of symmetric jump processes.

We first recall some known results on LILs of  stable processes, which are related to the topics of our paper.
Let $X:=(X_t)_{t\ge0}$ be a strictly $\beta$-stable process on ${\mathbb R}$ in the sense of Sato \cite[Definition 13.1]{sato} with $0<\beta<2$ and
$\nu((0,\infty))> 0$ for the L\'evy measure $\nu$ of $X$.
Then the following facts are well-known (see \cite[Propositions 47.16 and 47.21]{sato}).
\begin{proposition}\label{strstabq}
$(1)$ Let $h$ be a positive continuous and increasing function on $(0,\delta]$ for some $\delta>0$. Then
\[
\limsup_{t\to 0}\frac {|X_t|}{h(t)}= 0\quad\mbox{ a.s. ~~or }~ =\infty\quad \mbox{ a.s.}
\]
according to $\int_0^\delta h(t)^{-\beta}dt <\infty$ or $=\infty$, respectively.\\
$(2)$ Assume that
$X$  is not a subordinator.
Then there exists a constant $c\in (0,\infty)$ such that
\[
\liminf_{t\to 0} \frac{\sup_{0<s\le t} |X_s|}{(t/\log|\log t|)^{1/\beta}}
=c\quad\mbox{ a.s.}.
\]
\end{proposition}
Proposition \ref{strstabq}(1) was obtained by Khinchin in
\cite{Kh1}. A multidimensional version of Proposition
\ref{strstabq}(2) was first proved by Taylor in \cite{T}, and then a
refined version of Proposition \ref{strstabq}(2) for (non-symmetric)
L\'evy processes was established by Wee in \cite{Wee1}. We refer the
reader to \cite{ADS, BM, BMM, Sav} and the references therein.
Recently the results in Proposition \ref{strstabq} have been
extended to some class of Feller processes (see \cite{KS} and the
references therein).

When $\beta>1$, a local time of $X$  exists, and various LILs for
the local time are known. In the next result we concentrate on a
symmetric $\beta$-stable process $X$ on ${\mathbb R}$.
\begin{proposition}\label{strstabq2}
Assume $\beta \in (1,2)$. Then, there exist a local time $\{l(x,t): x\in {\mathbb R}, t>0\}$
and constants $c_1,c_2\in(0,\infty)$ such that
\begin{align}\label{slllil1}
\limsup_{t\to\infty} \frac{\sup_y l(y,t)}{t^{1-1/\beta}
(\log\log t)^{1/\beta}}=c_1\quad \mbox{ a.s.}
\end{align}
and
\begin{align}\label{slllil2}
\liminf_{t\to\infty} \frac{\sup_y l(y,t)}{t^{1-1/\beta}
(\log\log t)^{-1+1/\beta}}
=c_2\quad \mbox{ a.s.}.
\end{align}
\end{proposition}
 In \cite{Gr}   Griffin showed that  \eqref{slllil2} holds, and in \cite{Wee}  Wee has extended
 \eqref{slllil2}
 to a large class of L\'evy processes.
As applications of the large deviation method, \eqref{slllil1} was
proved by Donsker and Varadhan in \cite{DV}. For the case of
diffusions, LILs for the local time have further considered on
metric measure spaces including fractals based on the large
deviation technique (see \cite{fst, BK}); however, the corresponding
work for (non-L\'evy) jump processes is still not available. It
would be very interesting to see to what extent the above results
for L\'evy processes are still true for general jump processes,
e.g.\ see \cite[p.\ 306]{Xi}. Thus, we are concerned with the
following;
\begin{question}
If the generator of the process $X$ is perturbed so that the corresponding process with new generator is no longer a L\'evy process, do
the results in Propositions $\ref{strstabq}$ and $\ref{strstabq2}$  still hold$?$
\end{question}
In this paper, we consider this problem for a large class of symmetric Markov jump processes on metric measure spaces via heat kernel estimates.

\medskip

In order to explain our results
explicitly, let us first give the framework.
Let $(M,d)$ be a locally compact, separable and connected metric space, and let $\mu$ be a Radon measure on $M$ with full support.
We assume that $B(x,r)$ is relatively compact for all $x\in M$ and $r>0$.
Let $(\mathscr{E}, \mathscr{F})$ be a symmetric regular Dirichlet form on $L^2(M,\mu)$.  By the Beurling-Deny formula, such form can be decomposed
into three terms --- the strongly local term, the pure-jump term and the killing term (see \cite[Theorem 4.5.2]{FOT}).
Throughout this paper, we consider the form that consists of the pure-jump term only; namely
there exists a symmetric Radon measure $n(\cdot,\cdot)$ on $M\times M\setminus \textrm{diag}$,
 where \textrm{diag} denotes the diagonal set $\{(x, x): x\in M\}$,  such that
 \begin{equation}\label{df}\begin{split}\mathscr{E}(u,v)=&\int_{M\times M\setminus \textrm{diag}}\!\!(u(x)-u(y))(v(x)-v(y))\, n(dx,dy)\end{split}\end{equation} for all $u,v\in \mathscr{F}\cap C_c(M)$.
We denote the associated Hunt process by $X=(X_t,t\ge0; \Pp^x, x\in M; \mathscr{F}_t, t\ge0)$.
Then there is a properly exceptional set $\mathscr{N}\subset M$ such that the associated Hunt process is uniquely determined up to any starting point outside $\mathscr{N}$.
Let $(P_t)_{t\ge0}$ be the semigroup corresponding to $(\mathscr{E}, \mathscr{F})$,
and set $\R_+=(0, \infty)$. A heat kernel (a transition density) of $X$ is a non-negative symmetric measurable function $p(t,x,y)$ defined on $\R_+\times M\times M$ such that
$$P_tf(x)=\int_M p(t,x,z)f(z)\,\mu(dz),~\,\,\,~p(t+s,x,y)=\int_M p(t,x,z)p(s,z,y)\,\mu(dz),$$
for any Borel function $f$ on $M$, for all $s,t>0$,
all $x \in M \setminus \mathscr{N}$
 and $\mu$-almost all $y\in M$.
\medskip

We will use ``$:=$" to denote a
definition, which is read as ``is defined to be".
For $a, b\in {\mathbb R}$,
$a\wedge b:=\min \{a, b\}$ and $a\vee b:=\max\{a, b\}$.
The following is our main theorem for the case of
$\beta$-stable like processes on $\alpha$-sets.

\begin{theorem}\label{stable-like}
{\bf[$\beta$-stable-like processes on $\alpha$-sets]}
Let $(M,d,\mu)$ be as above.
Consider a symmetric regular Dirichlet form $(\mathscr{E}, \mathscr{F})$ on $L^2(M,\mu)$ that has
the transition density function $p(t,x,y)$.
We assume $\mu$ and $p(t,x,y)$ satisfy that
\begin{itemize}
 \item[(i)] there is a constant $\alpha>0$ such that \begin{equation}\label{eq:d-set}  c_1r^\alpha\le \mu(B(x,r))\le c_2r^\alpha,\quad x\in M, r>0,\end{equation}
  \item[(ii)] there also exists a constant $\beta>0$ such that for all $x,y\in M$ and $t>0$,
 \begin{align}\label{e:hke}
 c_3\left( t^{-\alpha/\beta}\wedge \frac{t}{d(x,y)^{\alpha+\beta}} \right)\le p(t,x,y)\le c_4 \left( t^{-\alpha/\beta}\wedge \frac{t}{d(x,y)^{\alpha+\beta}} \right).
 \end{align}
 \end{itemize}

Then, we have the following statements.
 \begin{itemize}
\item[(1)] If $\varphi$ is a strictly increasing function  on $(0,1)$ satisfying
\begin{equation}\label{eq:int01}
\int_0^1 \frac{1}{\varphi(s)^\beta} \,ds <\infty \quad \textrm{(resp.} =\infty\textrm{)},
\end{equation}
then
\begin{equation}\label{eq:int-conc}\limsup_{t\to0} \frac{\sup_{0<s\le t} d(X_s,x)}{\varphi(t)}=0\quad \textrm{(resp.} =\infty\textrm{)},~~~\quad\,
\Pp^x\mbox{-a.e. } \omega,~~\forall x\in M.
\end{equation}
Similarly, if $\varphi$ is defined on $(1,\infty)$ and the integral in \eqref{eq:int01} is
over $[1,\infty)$, then \eqref{eq:int-conc} holds for $t\to \infty$ instead of $t\to 0$.

\item[(2)] There exist constants $c_5,c_6\in(0,\infty)$ such that  for
all $x\in M$ and $\Pp^x\mbox{-a.e.}$,
\begin{align*}\liminf_{t\to0} \frac{\sup_{0<s\le t} d(X_s,x)}{(t/\log|\log t|)^{1/\beta}}=c_5,~~\,\,
\liminf_{t\to \infty} \frac{\sup_{0<s\le t} d(X_s,x)}{(t/\log\log t)^{1/\beta}}=c_6.
\end{align*}
\item[(3)] Assume $\alpha<\beta$. Then, there exist a local time $\{l(x,t): x\in M, t>0\}$
and constants $c_7,c_8,c_9,c_{10}\in(0,\infty)$ such that for all
$x\in M$ and $\Pp^x\mbox{-a.e.}$,
\begin{align*}
\limsup_{t\to\infty} \frac{
\sup_y l(y,t)}{t^{1-\alpha/\beta}
(\log\log t)^{\alpha/\beta}}=c_7, ~~\,\,
\liminf_{t\to\infty} \frac{
\sup_y l(y,t)
}{t^{1-\alpha/\beta}
(\log\log t)^{-1+\alpha/\beta}}
=c_8,\\
~~~~~~~~\,\,\,\,~\limsup_{t\to\infty} \frac{R(t)}{{t^{\alpha/\beta}
(\log\log t)^{1-\alpha/\beta}}}=c_9, ~~\,\,
\liminf_{t\to\infty} \frac{R(t)}{{t^{\alpha/\beta}
(\log\log t)^{-\alpha/\beta}}}=c_{10},
\end{align*}
where $R(t):=\mu(X([0,t]))$ is the range of the process $X$.
\end{itemize}
\end{theorem}
Note that in \cite{CK}, \eqref{e:hke} is proved for stable-like processes, that is
\begin{equation}\label{eq:DFshape}
\mathscr{E}(u,v)=\int_{M\times M\setminus \{x=y\}}\!\!(\widetilde u(x)-\widetilde u(y))(\widetilde v(x)-\widetilde v(y))\, n(dx,dy),\quad
\forall u,v\in {\mathscr{F}},
\end{equation}
where $\widetilde u$ is a quasi-continuous version of $u\in {\mathscr{F}}$,
and the L\'evy measure $n(\cdot,\cdot)$ satisfies
\[
c_1'\frac {\mu(dx)\mu(dy)}{d(x,y)^{\alpha+\beta}}\le n(dx,dy)\le c_2'\frac {\mu(dx)\mu(dy)}{d(x,y)^{\alpha+\beta}},
\]
for $\beta\in (0,2)$.
$\beta$-stable-like processes are perturbations of $\beta$-stable processes, and clearly
they are no longer L\'evy processes in general. Stable-like processes are analogues of uniformly elliptic divergence forms in the framework of jump processes.
-- We emphasize here that, in Theorem \ref{stable-like} above, we do not assume $\beta<2$ in general (see Example \ref{exfe4}).
Indeed, in this paper we will consider more general jump processes that include
jump processes of mixed types on metric measure spaces, which are given in Section  \ref{Sect5}.

\medskip

For the case of diffusions that enjoy the so-called sub-Gaussian
heat kernel estimates, LILs corresponding to Theorem \ref{stable-like} have been
established in \cite{BK,fst}.  However, since the proof uses Donsker-Varadhan's large deviation
theory for Markov processes, some self-similarity of the process is assumed in these papers (see \cite[(4.4)]{BK}
and \cite[(1.7)]{fst}).
In the present paper, we will not assume such a self-similarity on the process $X$. Instead we consider a family of scaling processes and take a (somewhat classical) \lq\lq bare-hands\rq\rq\, approach.

\bigskip

The remainder of the paper is organized as follows. In Section \ref{secn2}, we give
the assumptions on estimates of heat kernels
we will use, and present their consequences.
In Section \ref{Sect3}, we establish LILs for sample paths.
Section \ref{Sect4} is devoted to the LILs of maximums of local times and ranges of processes.
The LILs for jump processes of mixed types on metric measure spaces are given in Section \ref{Sect5} to illustrate the power of our results. Some of the proofs and technical lemmas are
left in Appendix \ref{Appx}.

Throughout this paper, we will use $c$, with or
without subscripts and superscripts,
to denote strictly
positive finite constants whose values are insignificant and may change from line to
line. We write $f\asymp g$ if there exist constants $c_1, c_2>0$ such that
$c_1g(x)\le f(x)\le c_2g(x)$ for all $x$.

\section{Heat Kernel Estimates and Their Consequences}\label{secn2}
Let $(M,d)$ be a locally compact, separable and connected metric space, and let $\mu$ be a Radon measure on $M$ with full support such that for any $x\in M$ and $r>0$,
\begin{equation}\label{e:BV}
C_*^{-1}V(r) \le \mu(B(x,r))\le C_* V(r),\end{equation} 
where $C_* \ge 1$ and $V:\R_+\to\R_+$ is a strictly increasing function satisfying that 
there exists a constants $c>1$ so that
\begin{equation}\label{v-d} V(0)=0, \quad V(\infty)=\infty \quad  \textrm{ and } \quad V(2r)\le cV(r)\quad\textrm{ for every }r>0.\end{equation}
Note that \eqref{v-d} is equivalent to the following: there exist constants $c,d>0$ such that
\begin{equation}\label{v-d-1} V(0)=0, \quad  V(\infty)=\infty \quad  \textrm{ and } \quad \frac{V(R)}{V(r)}\le c\Big( \frac{R}{r}\Big)^d\quad\textrm{ for all } 0<r<R.\end{equation}
Let $(\mathscr{E}, \mathscr{F})$ be a symmetric regular Dirichlet form on $L^2(M,\mu)$.
In this paper we will consider the following type of estimates for heat kernels:
 there exists
a properly exceptional set $\mathscr{N}$ and,  for given $T \in (0, \infty]$,
 there exist positive constants $C_1$ and $C_2$
such that  for
all $x\in M \setminus \mathscr{N}$,
$\mu$-almost all $y\in M$ and $t \in (0,T)$,
\begin{eqnarray}\label{a-two-sidedup}
p(t,x,y)\le C_1\bigg( \frac{1}{V(\phi^{-1}(t))} \wedge \frac{t}{V(d(x,y))\phi(d(x,y))}  \bigg),
\\
\label{a-two-sidedlow}
C_2\bigg( \frac{1}{V(\phi^{-1}(t))} \wedge \frac{t}{V(d(x,y))\phi(d(x,y))}  \bigg)
\le p(t,x,y),
\end{eqnarray} where $\phi:\R_+\to \R_+$ is a strictly increasing function.

We now state
the first set of our assumptions on heat kernels.

\begin{assumption}\label{assmp1}
There exists a transition density $p(t,x,y): \R_+ \times M \times M \to [0, \infty]$ of the semigroup of $(\mathscr{E}, \mathscr{F})$  satisfying
\eqref{a-two-sidedup} and \eqref{a-two-sidedlow} with $T=\infty$, and \eqref{v-d}.
\end{assumption}
\begin{assumption}\label{assmp11}
 $\phi(0)=0$, and there exist constants $c_0\in(0,1)$ and $\theta>1$ such that for every $r>0$
\begin{equation}\label{phi+w} \phi(r)\le c_0 \phi(\theta r).\end{equation}
\end{assumption} It is easy to see that under \eqref{phi+w}, $\lim\limits_{r\to\infty}\phi(r)=\infty$, and there exist constants $c_0,d_0>0$ such that $$ c_0\Big( \frac{R}{r}\Big)^{d_0}\le \frac{\phi(R)}{\phi(r)}\quad \textrm{ for all } 0<r<R,$$ e.g. the proof of \cite[Proposition 5.1]{GH11}.

In this section, we assume the above heat kernel estimates and discuss the consequences. Sometimes we only consider two-sided estimates on the heat kernel for short time. We say that Assumption \ref{assmp1} holds with $T<\infty$, if there exists a transition density $p(t,x,y): \R_+ \times M \times M \to [0, \infty]$ of the semigroup of $(\mathscr{E}, \mathscr{F})$  satisfying
\eqref{a-two-sidedup} and \eqref{a-two-sidedlow} with $T<\infty$, and \eqref{v-d}.
We emphasize that the constants appearing in the statements of this section only depend
on heat kernel estimates \eqref{a-two-sidedup} and \eqref{a-two-sidedlow}.

\medskip

Before we go on, let us note that
\eqref{a-two-sidedup} and \eqref{a-two-sidedlow} can be proved
in a rather wide framework.

\begin{theorem}{\rm (}\cite[Theorem 1.2]{CK1}{\rm )}
Let $(M,d,\mu)$ be a metric measure space given above with
$\mu(M)=\infty$. We assume that
 $\mu(B(x,r))\asymp V(r)$ for all $x\in M$ and $r>0$ where $V$ satisfies \eqref{eqn:univd*} below.
We also assume that there exist $x_0\in M$, $\kappa \in (0, 1]$ and an increasing sequence $r_n\to \infty$ as $n\to \infty$ so that for every $n\geq 1$, $0<r<1$ and $x\in \overline{B(x_0, r_n)}$,
there is some ball $B(y, \kappa r)\subset B(x, r) \cap \overline{B(x_0,r_n)}$.
Let $(\mathscr{E}, \mathscr{F})$ be a symmetric regular Dirichlet form on $L^2(M,\mu)$
such that $\mathscr{E}$ is given by \eqref{eq:DFshape} and the
L\'evy measure $n(\cdot,\cdot)$ satisfies
\begin{equation}\label{eq:Levyofe}
c_1\frac {\mu(dx)\mu(dy)}{V(d(x,y))\phi(d(x,y))}\le n(dx,dy)\le c_2\frac {\mu(dx)\mu(dy)}{V(d(x,y))\phi(d(x,y))}.
\end{equation}
Assume further that
 $\phi$
satisfies \eqref{assum-2} below
and that $\int_0^r(s/\phi(s))ds\le c_3r^2/\phi( r)$ for all $r>0$. Then there exists a jointly continuous
heat kernel $p(t,x,y)$
that enjoys the estimates \eqref{a-two-sidedup} and \eqref{a-two-sidedlow} with $T=\infty$.
\end{theorem}

\begin{remark}In \cite[Theorem 1.2]{CK1}, an additional assumption was made on the space $(M,d)$ such that it enjoys some scaling property (see \cite[p.\ 282]{CK1}). However, such
assumption can be removed by introducing a family of scaled distances as in \eqref{e:dmu} below
instead of assuming the existence of a family of scaled spaces, and by discussing similarly to the proof of Proposition \ref{thm:estlochold} below.
\end{remark}

\subsection{General case}
In this subsection, we state consequences of  Assumptions \ref{assmp1} and \ref{assmp11}. The proofs of next two propositions are given in Appendix  \ref{A1}.
We note that Proposition \ref{con} and its proof are due to \cite{CKW}.
\begin{proposition}\label{con}
If $p(t,x,y)$ satisfies \eqref{a-two-sidedlow} with $T=\infty$
 (in particular, if Assumption $\ref{assmp1}$ is satisfied), then
the process $X$ is conservative, i.e. for any $x\in M \setminus \mathscr{N}$ and $t>0$,
$$\int p(t,x,y)\,\mu(dy)=1.$$
\end{proposition}

\begin{proposition}\label{space} Let $p(t,x,y)$ satisfy
Assumptions $\ref{assmp1}$ and $\ref{assmp11}$ above. Then, we have 
$\textrm{Diam }(M)=\infty$ and  $\mu(M)=\infty$. 
Moreover,
there exist constants $c_1, c_2>0$, $d_2\geq d_1>0$
 such that
\begin{eqnarray}\label{eqn:univd*}
~~\mbox{   }~~\mbox{   }~\mbox{   }~~~
  c_1 \Big(\frac Rr\Big)^{d_1} \leq \frac{V
(R)}{V (r)} \ \leq \ c_2 \Big(\frac Rr\Big)^{d_2}~ \hbox{ for
every }~ 0<r<R<\infty.
\end{eqnarray}
\end{proposition}

\begin{proposition}\label{jump} Assume that the regular Dirichlet form $(\mathscr{E}, \mathscr{F})$ given by \eqref{df} enjoys the heat kernel $p(t,x,y)$ such that Assumption $\ref{assmp1}$ is satisfied. Then, the jump measure $n(dx,dy)$
satisfies \eqref{eq:Levyofe}.
\end{proposition}
For the assertion of $n(dx,dy)$, using the heat kernel estimates,  we can follow the proof of \cite[Theorem 1.2, (a)$\Rightarrow$ (c)]{BGK}.

\subsection{The case that $\phi$ satisfies the doubling property}

Throughout this subsection, we assume that $\phi$ satisfies the doubling property.
\begin{assumption}\label{assmp2}
There is a constant $c>1$ so that
\begin{equation}\label{doubphiw0}
\phi (2r)\le c \phi (r) \qquad \hbox{ for every }  r>0.
\end{equation}
\end{assumption}
Note that, $\eqref{doubphiw0}$ implies that for any $\theta>1$ there exists $c_0=c_0(\theta)>1$ such that for every $r>0$, $\phi(\theta r)\le c_0 \phi(r).$
If Assumptions \ref{assmp11} and \ref{assmp2} are satisfied, then it is easy to see (also see the proof of \cite[Proposition 5.1]{GH11}) that
$\phi$ satisfies the following inequality
 \begin{eqnarray}\label{assum-2}c_3\Big(\frac{R}{r}\Big)^{d_3}\le \frac{\phi(R)}{\phi(r)}\le c_4\Big(\frac{R}{r}\Big)^{d_4}\end{eqnarray} for all $0<r\le R$ and some positive constants $c_i, d_i (i=3,4)$.

In this subsection, we state consequences of
Assumptions \ref{assmp1}, \ref{assmp11} and \ref{assmp2}.
  The proofs
 of   Propositions  \ref{p:Holder_estimates}, \ref{w-ndle} and \ref{lemma-inf}
  in this subsection  are also given in Appendix  \ref{A1}.

We first prove the H\"older estimates for $p(t,x,y)$. As a result, under Assumptions \ref{assmp1}, \ref{assmp11} and  \ref{assmp2}, even in the case that Assumption \ref{assmp1} holds with $T<\infty$ and that the process $X$ is conservative,
the property exceptional set $\mathscr{N}$ can be taken to be the empty set, and so \eqref{a-two-sidedup} and \eqref{a-two-sidedlow} hold for all $x,y \in M$ and $t>0$.
We will frequently use this fact without explicitly mentioning it.

\begin{proposition}\label{p:Holder_estimates}
Suppose Assumptions $\ref{assmp1}$, $\ref{assmp11}$ and $\ref{assmp2}$ hold.
Then there exist constants $\theta \in (0,1]$ and $c>0$ such that for all $t\ge s >0$ and $x_i, y_i\in M$ with $i=1,2$
\begin{align}\label{e:Holder}
&|p(t,x_1,y_1)- p(s,x_2,y_2)| \nonumber\\
&\le\,  \frac{c}{V(\phi^{-1}(s))\phi^{-1}(s)^\theta} \left( \phi^{-1}(t-s)+d(x_1,x_2)+d(y_1,y_2)\right)^\theta.
\end{align}
In particular,
for all $t>0$ and  $x_i, y_i\in M$ with $i=1,2$
\begin{align}\label{e:Holder2}
|p(t,x_1,y_1)- p(t,x_2,y_2)| \le  \frac{c}{V(\phi^{-1}(t))} \left(\frac{d(x_1,x_2)+d(y_1,y_2)}{\phi^{-1}(t)}\right)^\theta.\end{align}

Furthermore, \eqref{e:Holder} and \eqref{e:Holder2} still hold true for any $0<s<t\le T$, if Assumptions $\ref{assmp11}$ and $\ref{assmp2}$ are satisfied,  Assumption $\ref{assmp1}$ only holds with $T<\infty$ and the process $X$ is conservative.
\end{proposition}

Using Proposition \ref{p:Holder_estimates}, we can get
\begin{theorem}[{\bf Zero-One Law for Tail Events}]\label{t:01tail}
Let $p(t,x,y)$ satisfy Assumptions $\ref{assmp1}$, $\ref{assmp11}$ and $\ref{assmp2}$ above, and let $A$ be a tail event. Then, either $\Pp^x(A)$ is $0$ for all $x$ or else it is $1$ for all $x \in M$.
\end{theorem}

For an open set $D$, we define
\begin{equation}
\label{e:hkos1}
p^D(t,x,y)  :=  p(t,x,y) - \Ee^x \big( \>p(t - \tau_D,
X_{\tau_D},y) : \tau_D < t\big),\quad
t>0, x,y \in D
\end{equation}
where $\tau_{D}:=\inf\{s>0: X_s\notin D\}.$
Using the strong Markov property of $X$, it is easy to verify that
$p^D(t,x,y)$ is the transition density for $X^D$,
the subprocess of $X$ killed upon leaving an open set $D$.
$p^D(t,x,y)$ is also called
the Dirichlet heat kernel of the process $X$ killed on exiting $D$.
The following two statements present a lower bound for the near diagonal estimate of Dirichlet heart kernels
and detailed controls of the distribution of the maximal process.

\begin{proposition}\label{w-ndle}
If Assumptions $\ref{assmp1}$, $\ref{assmp11}$ and $\ref{assmp2}$ hold, then there exist constants $\delta_0, c_0>0$ such that for any $x\in M$ and $r>0$,
\begin{equation}\label{e:pDlower}
p^{B(x,r)}(\delta_0 \phi (r), x',y')\ge c_0 V(r)^{-1},\quad x',y'\in B(x,r/2).
\end{equation}
Furthermore, if Assumptions $\ref{assmp11}$ and $\ref{assmp2}$ are satisfied, Assumption $\ref{assmp1}$ only holds for $T<\infty$ and the process $X$ is conservative, then \eqref{e:pDlower} holds for all $x\in M$ and $r\ge 0$ with $\delta_0 \phi (r)\in (0,T)$.
\end{proposition}

\begin{proposition}\label{lemma-inf}
If Assumptions $\ref{assmp1}$, $\ref{assmp11}$ and $\ref{assmp2}$ hold, then there exist some constants $c_0>0$ and $a^*_1, a^*_2\in (0,1)$ such that for
all $x\in M$, $r>0$ and $n\ge 1$,
\begin{equation}\label{e:ppnn}{a^*_1}^n\le \Pp^x(\sup_{0\le s\le c_0 n \phi(r)} d(X_s, x)\le r)\le {a^*_2}^n.\end{equation}
Furthermore, if Assumptions $\ref{assmp11}$ and $\ref{assmp2}$ are satisfied, Assumption $\ref{assmp1}$ only holds for $T<\infty$ and the process $X$ is conservative, then \eqref{e:ppnn} holds for all $x\in M$,  $n\ge 1$ and $r >0$ with $c_0n \phi(r)\le T$.
\end{proposition}

Let us introduce a space-time process $Z_s=(V_s, X_s)$, where $V_s=V_0+s$.  The law of the space-time process $s\mapsto Z_s$ starting from $(t,x)$ will be denoted by $\Pp^{(t,x)}$.
For any $r,t, \delta>0$ and $x \in M$, we define
$$
Q_\delta(t,x,r)=[t,t+\delta\phi(r)]\times B(x,r).
$$
We say that a non-negative Borel measurable function $h(t,x)$ on $[0,\infty)\times M$ is parabolic in a relatively open subset $D$ of $[0,\infty)\times M$, if for every relatively compact open subset $D_1\subset D$,
$h(t,x)=\Ee^{(t,x)} h(Z_{\hat \tau_{D_1}})$ for every $(t,x)\in D_1$, where $\hat \tau_{D_1}=\inf\{s>0: Z_s\notin D_1\}.$

We now state the following parabolic Harnack inequality.
\begin{proposition}\label{thm:PHI_hw}
 Assume that Assumptions $\ref{assmp1}$, $\ref{assmp11}$ and $\ref{assmp2}$ hold.
  For every $0<\delta <1$, there exists $c_1>0$ such that for every $z\in M$, $R>0$ and every non-negative function $h$ on $[0,\infty)\times M$, that is parabolic on $[0, 3\delta \phi(R)]\times B(z,2R)$,
  $$\sup_{(t,y)\in Q_\delta(\delta \phi(R), z,R)} h(t,y)\le c_1 \inf_{y\in B(z,R)} h(0,y).$$
\end{proposition}
By Assumptions $\ref{assmp1}$, $\ref{assmp11}$ and $\ref{assmp2}$ and Proposition \ref{jump}, the density $J(x,y)$ of the jump measure $n(dx,dy)$ satisfies the following upper jump smoothness
 ({\bf UJS}): there exists a constant $c_1>0$ such that
for $\mu$-a.e. $x, y\in M$,
$$
 J(x,y) \le \frac{c_1}{V(r)}
\int_{B(x,r)} J(z,y)\,\mu(dz) \quad \hbox { whenever $r\le \frac 12
d(x,y)$} . $$ Noting that $J(x,y)=\lim\limits_{r\to0}
\frac{1}{\mu(B(x,r))} \int_{B(x,r)} J(z,y)\,\mu(dz)$ for $\mu$-a.e.
$x, y\in M$, ({\bf UJS}) is a kind of smooth assumption on the upper
bound of jump kernel $J(x,y)$. Let $c$ be the constant in Assumption
$\ref{assmp2}$, and $c_0\in(0,1)$ be the constant such that for
almost all $x\in M$ and $r>0$,
\begin{equation}\label{diff}\Pp^x(\tau_{B(x,r/2)}\le c_0\phi(r))\le 1/2,\end{equation}
see e.g. \eqref{exit} below. Since the density $J(x,y)$ of the jump measure $n(dx,dy)$ satisfies ({\bf UJS}),
Proposition \ref{thm:PHI_hw} can be proved by following
the arguments of \cite[Theorem 4.12]{CK1} and \cite[Theorem 5.2]{CKK3}.
See \cite[Appendix B]{CK1} and \cite[Section 5]{CKK3} for more details. In fact, as explained in the first paragraph of \cite[Theorem 5.2]{CKK3}
one can  first consider the case that $h$ is non-negative and bounded
on $[0, \infty)\times F$ and establish the result for $\delta \leq c_0/c$.
Once this is done, one can extend it to all $\delta <1$ and  any non-negative parabolic function (not necessarily bounded)  by a simple chaining argument and
the argument in the step 3 of the proof of  \cite[Theorem 5.2]{CKK3}, respectively.

\section{Laws of the Iterated Logarithm for Sample Paths}\label{Sect3}
In this section, we discuss LILs for sample paths of the process $X$.
Instead of assuming full heat kernel estimates as in Assumption \ref{assmp1}, we give
the estimates that are needed in each statement.
Throughout this paper (except Proposition \ref{space-c} below), we will always assume that
the reference measure $\mu$ satisfies 
the uniform volume doubling property in \eqref{e:BV} and that $V$ is a strictly increasing function that satisfies \eqref{v-d}.

\subsection{Upper bound for limsup behavior}
In this subsection we assume that the heat kernel $p(t,x,y)$ on
$(M,d,\mu)$ satisfies the following upper bound estimate for all $x
\in M\setminus\mathscr{N} $, $\mu$-almost all $y\in M$ and all $t
\in (a,b)$ with $a<b$,
\begin{equation}\label{upper-1} p(t,x,y)\le \frac{C \, t}{V(d(x,y))\phi(d(x,y))},\end{equation}
where $C>0$, and $\phi:\R_+\to \R_+$ is a strictly increasing
functions satisfying \eqref{assum-2}.

 \begin{theorem}\label{sup-1} Assume that the process $X$ is conservative. Then the following statements hold.

\begin{itemize}
\item[(1)] If  $
a=0$ and $\varphi$ is an increasing function on $(0,1)$ such that
\begin{align}\label{e:asump-sup1}\int_0^1
\frac{1}{\phi\big({\varphi(t)}\big)}\,dt<\infty,\end{align} then
$$
\limsup_{t\to0} \frac{\sup_{0\le s\le t} d(X_s, x)}{\varphi(t)}=0,
 ~~~\quad\,
\Pp^x\mbox{-a.e. } \omega,~~\forall
x\in M \setminus \mathscr{N}.
$$
\item[(2)] If
$b=\infty$ and $\varphi$ is an increasing function on on
$(1,\infty)$ such that
\begin{align*}\label{e:asump-sup2}\int_1^\infty
\frac{1}{\phi\big({\varphi(t)}\big)}\,dt<\infty,\end{align*} then
$$\limsup_{t\to \infty} \frac{\sup_{0\le s\le t} d(X_s,
x)}{\varphi(t)}=0,
 ~~~\quad\,
\Pp^x\mbox{-a.e. } \omega,~~\forall
x\in M \setminus \mathscr{N}.$$
\end{itemize}

 \end{theorem}

\begin{proof} We only prove (1), since (2) can be verified similarly.
Let us first check that
there is a constant $c_1>0$ such that for all $x\in M \setminus \mathscr{N}$, $r>0$ and $t\in(0,b)$,
\begin{equation}\label{point}\int_{B(x,r)^c}p(t,x,z)\,\mu(dz)\le  \frac{c_1t}{\phi(r)}.\end{equation}  If $t\ge \phi(r)$, then
the right hand side of \eqref{point} is greater than $1$ by taking $c_1>1$, so
we may assume
 that $t\le \phi(r)$. Without loss of generality, we also assume that $b=1$. It follows from \eqref{upper-1} and the increasing property of $V$ that, for
all $x\in M \setminus \mathscr{N}$, $\mu$-almost all  $z\in M$ with $d(x,z)\ge s$ and each  $t\in (0,1)$,
 $$p(t,x,z)\le \frac{C t}{V(s)\phi(s)}.$$
This upper bound,  along with  the uniform volume doubling property of $\mu$ (e.g.\  
\eqref{e:BV} and \eqref{v-d-1})
and \eqref{assum-2},
 yields that
 \begin{align*}
\int_{B(x,r)^c} p(t,x,z)\,\mu(dz)&\le \sum_{k=0}^\infty \int_{B(x,\theta^{k+1} r)\setminus B(x,\theta^k r)} p(t,x,z)\,\mu(dz)\\
&\le \sum_{k=0}^\infty \frac{C}{V(\theta^kr)} \frac{ t}{\phi(\theta^kr)}\mu \Big(B(x,\theta^{k+1} r)\setminus B(x,\theta^k r)\Big)\\
&\le \sum_{k=0}^\infty  \frac{c_2 V(\theta^{k+1}r)}{V(\theta^kr)}\frac{ t}{\phi(\theta^kr)}
\le c_3\sum_{k=0}^\infty c_0^{k}\frac{ t}{\phi(r)}
\le  \frac{ c_4t}{\phi(r)}. \end{align*}

 Recall that $\tau_{B(x,r)}=\inf\{t>0: X_t\notin B(x,r)\}.$
    By \eqref{point} and the strong Markov property and the conservativeness of $X$,  for
   all $x\in M \setminus \mathscr{N}$, $t\in(0,1)$ and $r>0$,
   \begin{equation}\label{exit}\begin{split} \Pp^x&(\tau_{B(x,r)}\le t)\\
   =&\Pp^x(\tau_{B(x,r)}\le t, X_{2t}\in B(x,r/2)^c)+\Pp^x(\tau_{B(x,r)}\le t, X_{2t}\in B(x,r/2))\\
   \le &\Pp^x(\tau_{B(x,r)}\le t, d(X_{2t},x)\le r/2)+\Pp^x(d(X_{2t},x)\ge r/2)\\
   \le &\Pp^x(\tau_{B(x,r)}\le t, d(X_{2t},X_{\tau_{B(x,r)}})\ge r/2)+ \frac{ 2 c_1 t}{\phi(r/2)}\\
   \le &\sup_{s\le t, d(z,x)\ge r} \Pp^z (d(X_{2t-s},z)\ge r/2)+ \frac{ 2c_1t}{\phi(r/2)}
   \le  \frac{ c_5 t}{\phi(r/2)}.\end{split}
 \end{equation}
 (Note that the conservativeness is used in the
equality above. Indeed, without the assumption of the conservativeness,
there must be an extra term $$\Pp^x(\tau_{B(x,r)}\le t, \zeta\le 2t)$$
in the right hand side of the equality above, where $\zeta$ is the
lifetime of the process $X$.)

Set $s_k=2^{-k-1}$ for all $k\ge 1$. By \eqref{exit}, we have that, for all $x\in M \setminus \mathscr{N}$
\begin{align*}
\Pp^x(\sup_{0<s\le s_k}d(X_s, x)\ge 2\varphi(s_{
k}))= \Pp^x(\tau_{B(x,2\varphi(s_{
k}))}\le s_{k})\le  \frac{ c_5 s_{k}}{\phi( \varphi(s_{k+1}))} .\end{align*}
By the assumption
\eqref{e:asump-sup1} and the Borel-Cantelli lemma,
 $$\Pp^x(\sup_{0<s\le s_k}d(X_s, x)\le 2\varphi(s_{k}))
 \textrm{ except finite }
 k\ge 1)=1,$$
 which implies that $$
\limsup_{t\to0} \frac{\sup_{0\le s\le t} d(X_s, x)}{\varphi(t)}\le 2,
 ~~~\quad\,
\Pp^x\mbox{-a.e. } \omega,~~\forall
x\in M \setminus \mathscr{N}.
$$ Therefore, the required assertion follows by considering $\varepsilon\varphi(r)$ for small $\varepsilon>0$ instead of $\varphi(r)$ and using \eqref{assum-2}.  \end{proof}

\begin{remark}
 From \eqref{point}, one can easily get similar statements for the limsup
  behavior of $d(X_t,x)$ for both $t\to0$ and $t\to\infty.$
 \end{remark}

\subsection{Lower bound for limsup behavior}

We begin with the assumption that the heat kernel $p(t,x,y)$ on $(M,d,\mu)$ satisfies the following
off-diagonal lower bound estimate: there are constants $a, C>0$ such that for
every $x\in M \setminus \mathscr{N}$,  $\mu$-almost all $y\in M$ and all $t\in(a,\infty)$,
\begin{equation}\label{lower-1} p(t,x,y)\ge  \frac{C \, t}{V(d(x,y))\phi(d(x,y))}, \quad d(x,y)\ge \phi^{-1}(t), \end{equation}
where $V$ and $\phi$ are strictly increasing functions satisfying \eqref{eqn:univd*} and \eqref{doubphiw0}, respectively.
The statement below presents lower bound for the limsup behavior of maximal process for $t\to \infty$.
\begin{theorem}\label{sup-2}
Let $p(t,x,y)$ satisfy the lower bound estimate \eqref{lower-1} above.
If $\varphi$ is an increasing function  on $(1,\infty)$
satisfying \begin{equation}\label{sup-2-con}\int_1^\infty  \frac{1}{\phi(\varphi(t))}\,dt=\infty,\end{equation}
then
for all
$x\in M \setminus \mathscr{N}$ \begin{align}
\label{e:l_bd_limsup1}\limsup_{t\to \infty} \frac{ \sup_{0<s\le t}d(X_s,x) }{\varphi\big({t}\big)
  } = \limsup_{t\to \infty} \frac{ d(X_t,x) }{\varphi(t)
}= \infty,\quad\,
\Pp^x\mbox{-a.e. }
\omega.
\end{align}
\end{theorem}

\begin{proof}
Without loss of generality, we can assume that
$a=1$ and $\phi(1)=1$.
First, choose $r_0\ge2$ such that $r_0^{-d_1}<c_1$, where $d_1$ and $c_1$ are constants given in \eqref{eqn:univd*}. By
\eqref{eqn:univd*} and \eqref{doubphiw0}, we have that for all $s \ge 1$
\begin{align*}\int_{r\ge s} \frac{1}{V(r)\phi(r)}\,dV(r)&=\sum_{k=0}^\infty \int_{r\in [r_0^ks, r_0^{k+1}s)}\frac{1}{V(r)\phi(r)}\,dV(r)\\
&\ge \sum_{k=0}^\infty \frac{V(r_0^{k+1}s)-V(r_0^{k}s)}{V(r_0^{k+1}s)\phi(r_0^{k+1}s)}\\
&\ge \left(1-\frac{1}{c_1r_0^d}\right)\sum_{k=0}^\infty \frac{1}{\phi(r_0^{k+1}s)}\\
&\ge \frac{1}{c_0}\left(1-\frac{1}{c_1r_0^d}\right)\sum_{k=0}^\infty c^{-(1+\log_2 r_0)(k+1)} \frac{1}{\phi(s)}\\
&=: c_2 \frac{1}{\phi(s)}.\end{align*}
In particular,
\begin{equation}\label{sup-2-1}\inf_{t\ge 1} \int_{r\ge \phi^{-1}(t)} \frac{ t}{V(r)\phi(r)}\,d V(r)>0,\end{equation}
and by \eqref{sup-2-con},
\begin{equation}\label{sup-2-2}\int_1^\infty \, dt \int_{r\ge \varphi(t)} \frac{ 1}{V(r)\phi(r)}\,d V(r) =\infty.\end{equation}

For any $k\ge 1$, set $B_k=\{ d(X_{2^{k+1}},X_{2^k}) \ge
\varphi (2^{k+1}) \vee \phi^{-1}(2^{k+1})
\}.$
Then for every $x\in M \setminus \mathscr{N}$
 and $k\ge 1$, by the Markov property,
\begin{align*}
\Pp^x (B_k|\mathscr{F}_{2^{k}})\ge &\inf_{z}\Pp^z(d(X_{2^k}, z)\ge
\varphi (2^{k+1}) \vee \phi^{-1}(2^{k+1}))\\
\ge& C\int_{r\ge \varphi (2^{k+1}) \vee \phi^{-1}(2^{k+1})} \frac{ 2^k}{V(r)\phi(r)}\,d V(r).
\end{align*}

If there exist infinitely many $k\ge 1$ such that $\varphi(2^{k+1})\le \phi^{-1}(2^{k+1})$, then, by \eqref{sup-2-1}, for infinitely many $k\ge 1$,
\begin{align*}\Pp^x (B_k|\mathscr{F}_{2^{k}})\ge & C\int_{r\ge\phi^{-1}(2^{k+1})} \frac{ 2^k}{V(r)\phi(r)} \,d V(r)\\
\ge & \frac{C}2\inf_{t\ge 1} \int_{r\ge\phi^{-1}(t)} \frac{ t}{V(r)\phi(r)}\,d V(r)=:c_3>0\end{align*} and so
\begin{equation}\label{sup-2-3}\sum_{k=1}^\infty \Pp^x (B_k|\mathscr{F}_{2^{k}})=\infty.\end{equation} If there is $k_0\ge1$ such that for all $k\ge k_0$, $\varphi(2^{k+1})>\phi^{-1}(2^{k+1})$, then
$$\Pp^x (B_k|\mathscr{F}_{2^{k}})\ge C\int_{r\ge \varphi(2^{k+1})} \frac{2^k}{V(r)\phi(r)}\,d V(r)=
\frac{C}2\int_{r\ge \varphi(2^{k+1})} \frac{2^{k+1}}{V(r)\phi(r)}\,d V(r).$$ Combining this with \eqref{sup-2-2}, we also get \eqref{sup-2-3}.
Therefore, by the second Borel-Cantelli lemma, $\Pp^x(\limsup B_n)=1$. Whence, for infinitely many $k\ge 1$,
  $$d(X_{2^{k+1}},x)\ge \frac{1}{2}( \varphi(2^{k+1})  \vee \phi^{-1}(2^{k+1}))
 $$
  or $$d(X_{2^k}, x)\ge \frac{1}{2}( \varphi(2^{k+1}) \vee \phi^{-1}(2^{k+1}))\ge \frac{1}{2}( \varphi(2^k) \vee \phi^{-1}(2^k)).$$
 In particular,
 $$\limsup_{t\to \infty}\frac{d(X_t,x)}{\varphi(t) \vee \phi^{-1}\big({t}\big)}\ge \limsup_{k\to \infty} \frac{d(X_{2^{k}},x)}{\varphi(2^k)\vee \phi^{-1}(2^k)}\ge \frac{1}{2}.$$

By the inequality above, we immediately get that  for all
$x\in M \setminus \mathscr{N}$
$$ \limsup_{t\to \infty} \frac{\sup_{0<s\le t} d(X_s,x) }{\varphi\big({t}\big)
  }\ge
\limsup_{t\to \infty} \frac{ d(X_t,x) }{\varphi\big({t}\big)
  } \ge
\frac{1}{2},\quad
\Pp^x\mbox{-a.e. }
\omega.$$
Therefore, \eqref{e:l_bd_limsup1} follows by considering $k\varphi(r)$ for large enough $k>1$ instead of $\varphi(r)$ and using \eqref{doubphiw0}.
 \end{proof}

To consider the lower bound for limsup behavior of maximal process for $t\to 0$, we need the following two-sided
off-diagonal estimate for the heat kernel $p(t,x,y)$ on $(M,d,\mu)$, i.e.
for every $x \in M\setminus \mathscr{N}$, $\mu$-almost all $y\in M$
and each $t \in (0,b)$ with some constant $b>0$,
 \begin{equation}\label{two-sided} \frac{C_1t}{V(d(x,y))\phi(d(x,y))}
 \le p(t,x,y)
 \le\frac{ C_2 t}{V(d(x,y))\phi(d(x,y))},\quad d(x,y)\ge \phi^{-1}(t),\end{equation}
 where $V$ and $\phi$ are strictly increasing functions satisfying \eqref{eqn:univd*} and \eqref{doubphiw0}, respectively.

\begin{theorem}\label{sup-3}  Let $p(t,x,y)$ satisfy two-sided
off-diagonal estimate
 \eqref{two-sided} above.
 If
$\varphi$ is an increasing function  on $(0,1)$
satisfying  \begin{equation}\label{sup-3-con}\int_0^1  \frac{1}{\phi(\varphi(t))}\,dt=\infty,\end{equation}
then for all $x\in M \setminus \mathscr{N}$,
 \begin{align}
 \label{e:lb_limsup1}\limsup_{t\to0} \frac{\sup_{0<s\le t} d(X_s,x)}{\varphi\big(t\big)}=
 \limsup_{t\to0} \frac{ d(X_t,x)}{\varphi\big(t\big)   }=\infty,~~~\qquad\,
\Pp^x\mbox{-a.e. } \omega. \end{align}
\end{theorem}

To prove Theorem \ref{sup-3},
we will adopt the following generalized Borel-Cantelli lemma.

\begin{lemma}\label{bc}$($\cite[Theorem 2.1]{P} or \cite[Theorem 1]{Y}$)$ Let $A_1$, $A_2$, $\ldots$ be a sequence of events satisfying conditions
$\sum_{n=1}^\infty \Pp(A_n)=\infty$ and
$\Pp(A_k\cap A_j)\le C\Pp(A_k)\Pp (A_j)$ for all $k,j>L$ such that $k\neq j$ and for some constants $C\ge1$ and $L$. Then,
$\Pp(\limsup A_n)\ge 1/C.$ \end{lemma}

\begin{proof}[Proof of Theorem $\ref{sup-3}$] For simplicity, we may and
will assume that $b=1$, $\phi(1)=1$ and $2^{-d_1}<c_1$, where $d_1$ and $c_1$ are constants given in \eqref{eqn:univd*}. Then, similar to the proof of Theorem \ref{sup-2}, under assumptions of the theorem, we have
\begin{equation}\label{sup-3-1}\inf_{t\in(0,1]} \int_{r\ge \phi^{-1}(t)} \frac{t}{V(r)\phi(r)}\,d V(r) >0,\end{equation}
and, by \eqref{sup-3-con},
 \begin{equation}\label{sup-3-2} \int_0^1 \, dt \int_{r\ge\varphi(t)}\frac{1}{V(r)\phi(r)}\,d V(r)=\infty.\end{equation}

 For some $t\in(0,1)$ and any $k\ge 1$, set $s_k=2^{-k}t$ and
$$A_k=\Big\{ d(X_{s_k},X_{s_{k+1}}) \ge \varphi (s_{k}) \vee \phi^{-1} (s_{k}) \Big\}.$$
 By the Markov property and the lower bound in \eqref{two-sided}, for
 all $x\in M \setminus \mathscr{N}$,
  \begin{align*}
\Pp^x (A_k)\ge &\inf_{z}\Pp^z(d(X_{s_{k+1}}, z)\ge \varphi(s_{k})   \vee \phi^{-1}(s_{k}))\nonumber\\
\ge&C_1\inf_{z}\int_{d(y,z)\ge\varphi(s_{k}) \vee \phi^{-1}(s_{k})}\frac{s_{k+1}}{V(d(z,y))\phi(d(z,y))}\mu(dy)\nonumber\\
\ge &c_2\int_{r\ge\varphi(s_{k}) \vee \phi^{-1}(s_{k})} \frac{s_{k}}{V(r)\phi(r)} \,dV(r)
=:c_2 c_{1,s_{k}}.\end{align*}
In particular, if $\varphi(\theta)\ge \phi^{-1}(\theta)$, then  \begin{align*}c_{1,\theta}
= &\int_{r\ge \varphi(\theta)}\frac{\theta}{V(r)\phi(r)}\,d V(r)
  ; \end{align*}  if $\varphi(\theta)\le \phi^{-1}(\theta)$, then  \begin{equation}\label{sup-3-3}\begin{split}c_{1,\theta}
   &= \int_{r\ge \phi^{-1}(\theta)}\frac{\theta }{V(r)\phi(r)}\,d V(r).\end{split}\end{equation}
Combining these two estimates above with \eqref{sup-3-1} and \eqref{sup-3-2} yields that
  $$\sum_{k=1}^\infty
  \Pp^x(A_k)=\infty.$$

On the other hand, for any $k<j$, by the Markov property and the upper bound for the heat kernel \eqref{two-sided},
 \begin{align*}\Pp^x (A_k\cap A_j)\le &\Ee^x\Big(\I_{A_j} \Pp^{X_{s_{k}}}\big(d(X_{s_{k+1}}, X_0)\ge \varphi(s_{k})\vee \phi^{-1} (s_{k})\big)\Big)\\
 \le &\Pp^x(A_j) \sup_{z}\Pp^z\big(d(X_{s_{k+1}}, z)\ge \varphi(s_{k})\vee \phi^{-1} (s_{k})\big)\\
 \le & c_3\Pp^x(A_j)c_{1,s_{k}}
 \le  c_3^2c_{1,s_{j}}c_{1,s_{k}}.\end{align*}
 From this and \eqref{sup-3-3}, we can easily see that there is a constant $C_0\ge 1$ such that
 $$ \Pp^x (A_k\cap A_j)\le C_0\Pp^x(A_k)\Pp^x(A_j).$$

 Therefore, according to Lemma \ref{bc}, $\Pp^x(\limsup A_n)\ge 1/C_0,$ which along with the Blumenthal 0-1 law implies that $\Pp^x(\limsup A_n)=1$. Whence, for infinitely many $k\ge 1$,
  $$d(X_{s_{k}},x)\ge \frac{1}{2} (\varphi(s_{k})
\vee  \phi^{-1}(s_{k}))
$$
  or
  $$d(X_{s_{k+1}}, x)\ge \frac{1}{2} (\varphi(s_{k})
  \vee  \phi^{-1}(s_{k}))
  \ge \frac{1}{2} \left(\varphi\left(s_{k+1}\right)\vee \phi^{-1}\left({s_{k+1}}\right)\right) .$$ In particular,
 $$\limsup_{t\to0}\frac{d(X_t,x)}{\varphi(t)   \vee \phi^{-1}(t) }\ge \limsup_{k\to \infty} \frac{d(X_{s_{k}},x)}{\varphi(s_{k})
  \vee  \phi^{-1}(s_{k})
 }\ge \frac{1}{2}.$$
Hence, \eqref{e:lb_limsup1} follows by considering $k\varphi(r)$ for large $k>1$ instead of $\varphi(r)$ and using  \eqref{doubphiw0}.
\end{proof}

\begin{remark}\label{rem26}
 The proof of Theorem \ref{sup-2} is
 based only
  on off-diagonal lower bound of the heat kernel estimate for long time, while in the proof of Theorem \ref{sup-3} explicit two-sided off-diagonal estimate of the heat kernel for small time is used.
Unlike the case of Theorem \ref{sup-2}, we do not know how to prove Theorem \ref{sup-3}
by using only the off-diagonal lower bound of the heat kernel estimate.
  \end{remark}

\subsection{
 Liminf laws of the iterated logarithm}

In this part, we discuss
Chung-type liminf laws of the iterated logarithm. To this end, we assume that  the heat kernel $p(t,x,y)$ on $(M,d,\mu)$ satisfies the following two-sided estimates with $T\in(0,\infty]$:
for every $x \in M\setminus \mathscr{N}$, $\mu$-almost all $y\in M$
and each $0<t<T$, \begin{equation}\label{two-sided-1}\begin{split} & C_1\bigg( \frac{1}{V(\phi^{-1}(t))} \wedge \frac{t}{V(d(x,y))\phi(d(x,y))}  \bigg)\le p(t,x,y),\\
  &p(t,x,y)\le C_2\bigg( \frac{1}{V(\phi^{-1}(t))} \wedge \frac{t}{V(d(x,y))\phi(d(x,y))}  \bigg),\end{split}\end{equation} where  $V$ and $\phi$ are strictly increasing functions satisfying \eqref{eqn:univd*} and \eqref{assum-2}
  respectively.

\begin{theorem}\label{inf1} Assume that the process $X$ is conservative.
Let $p(t,x,y)$ satisfy two-sided
estimate
 \eqref{two-sided-1} above with $0<T<\infty$.
Then there exists a constant $c\in (0,\infty)$ such that
$$\liminf_{t\to0} \frac{\sup_{0<s\le t} d(X_s,x)}{\phi^{-1}(t/\log|\log t|)}=c,~~~\quad\,
\Pp^x\mbox{-a.e. } \omega,~~\forall x\in M.$$ \end{theorem}

\begin{proof}
The following proof is based on the idea of proofs in \cite[Chapter 3]{Du} (see also the proof of \cite[Theorem 2]{KS}). Without loss of generality, we can assume that $T=1$, and $\mathscr{N}= \emptyset$ due to Proposition \ref{p:Holder_estimates}.

Let $(a_k)_{k\ge 1}$ be the sequence defined by $a_k=
\phi^{-1}(e^{-k^2})$ so that $\phi(a_k)= e^{-k^2}$.  For any
$k\ge1$, set $\lambda_k=\frac{2}{3|\log a^*_1|}\log(1+k)$,
$u_k=c_0\lambda_ke^{-k^2}$ and $\sigma_k=\sum_{i=k+1}^\infty u_i$,
where
 $c_0>0$ and $a^*_1 \in (0,1)$ are the constants in Proposition \ref{lemma-inf}.
We will prove that there are $\xi, c_1\in(0,\infty)$ such that for all $x\in M$
$$\Pp^x\left(\sup_{2a_{2m}\le r\le 2a_m} \frac{\tau_{B(x,r)}}{\phi(r)\log|\log \phi(r)|}\le \xi\right)\le c_1\exp(-m^{1/4}),\quad m\ge 1.$$

For $k\ge 1$, let $G_k=\big\{\sup_{\sigma_k\le s\le \sigma_{k-1}} d(X_s,X_{\sigma_k})>a_k\big\}$. By the Markov property, the conservativeness of the process $X$
and Proposition \ref{lemma-inf},
 for all $x\in M$,
\begin{align*} \Pp^x (G_k)&\le \sup_z\Pp^z\big(\sup_{0\le s\le u_k} d(X_s,z)> a_k\big)\\
&=1-\inf_z\Pp^z\big(\sup_{0\le s\le u_k} d(X_s,z)\le a_k\big)\\
&=1-{a^*_1}^{\lambda_k}= 1-(1+k)^{-2/3}\le \exp(-c_2k^{-2/3}).\end{align*}

For $k\ge 1$, let $H_k=\big\{  \sup_{0<s\le \sigma_k} d(X_s,x)>a_k\big\}.$ Then, for
all $x\in M$
and for all $k\ge 1$,
\begin{align*} \Pp^x (H_k)&\le  \frac{c_3\sigma_k}{\phi(a_k)}\le \frac{c_4 \sum_{i=1}^\infty e^{-(k+i)^2} \log (1+k+i) }{e^{-k^2}}\le c_5 e^{-k },\end{align*} where the first inequality follows from \eqref{exit} and the doubling property of $\phi$.

For $m\ge 1$, define $A_m=\bigcap_{k=m}^{2m} D_k$, where $D_k=\big\{\sup_{0<s\le \sigma_{k-1}} d(X_s,x)>2a_k\big\}.$ Since $D_k\subset G_k\cup H_k$,
$A_m\subset \left(\cap_{k=m}^{2m} G_k\right)\cup \left(\cup_{k=m}^{2m} H_k\right).$
By using the Markov property again, we find that for
all $x\in M$,
\begin{align*} \Pp^x (A_m)&\le \Pp^x(\cap_{k=m}^{2m} G_k)+\Pp^x (\cup_{k=m}^{2m} H_k)\\
&\le \prod_{k=m}^{2m} \exp(-
c_2 k^{-2/3})+ c_5 \sum_{k=m}^{2m}e^{-k}\le
c_6\exp(-m^{1/4}).\end{align*}
Therefore, \begin{align*}
c_6\exp(-m^{1/4})\ge &\Pp^x\Big(\bigcap_{k=m}^{2m} \Big\{ \frac{\sup_{0<s\le \sigma_{k-1}} d(X_s,x)}{2a_k}>1 \Big\}\Big)\\
=&\Pp^x\Big(\inf_{m\le k\le 2m} \frac{\sup_{0\le s\le \sigma_{k-1}} d(X_s,x)}{2a_k}>1\Big)\\
=&\Pp^x\Big(\sup_{m\le k\le 2m} \frac{\tau_{B(x,2a_k)}}{\sigma_{k-1}}<1\Big)
\ge\Pp^x(\sup_{m\le k\le 2m} \frac{\tau_{B(x,2a_k)}}{u_k}<1)\\
\ge& \Pp^x\Big(\sup_{2a_{2m}\le r\le 2a_m}\frac{\tau_{B(x,r)}}{ \phi(r)\log|\log \phi(r)|} \le \xi\Big)\end{align*} for some $\xi\in(0,\infty)$.
Using this equality, by the Borel-Cantelli lemma, we  conclude that
$$\limsup_{r\to0} \frac{\tau_{B(x,r)}}{ \phi(r)\log|\log \phi(r)|}\ge \xi.$$

On the other hand, with  $l_k:=\phi^{-1}(e^{-k})$ for $k\ge1$, we
have
\begin{align*} B_k:=\Big\{ \sup_{l_{k+1}\le r\le l_k} \frac{\tau_{B(x,r)}}{\phi(r) \log|\log \phi(r)|}\ge b\Big\}
\subset \Big\{{\tau_{B(x,l_k)}}\ge b e^{-1} {\phi(l_{k}) \log|\log \phi(l_k)|}\Big\}. \end{align*}
Taking $b=-4/\log a^*_2$
where $a^*_2 \in (0,1)$ is the constant in Proposition \ref{lemma-inf}, we know from
Proposition \ref{lemma-inf}
that  $\Pp^x(B_k)\le k^{-4/e}.$ Thus, by the Borel-Cantelli lemma again,
$$\limsup_{r\to0} \frac{\tau_{B(x,r)}}{\phi(r) \log|\log \phi(r)|}\in[\xi,b],$$ which implies that
$$\limsup_{r\to0} \frac{\tau_{B(x,r)}}{\phi(r) \log|\log \phi(r)|}=C,~~~\quad\,
\Pp^x\mbox{-a.e. } \omega,~~\forall x\in M,$$ for some constant $C>0$, also thanks to the Blumenthal 0-1 law.
The desired assertion follows from the equality above.
\end{proof}

For the behavior of liminf for maximal process with $t\to \infty$, we have the following conclusion similar to Theorem \ref{inf1}.
\begin{theorem}\label{inf2} Let $p(t,x,y)$ satisfy two-sided
estimate
 \eqref{two-sided-1}
 for all $t>0$, i.e. $T=\infty$.
Then there exists a constant $c\in (0,\infty)$ such that
$$\liminf_{t\to \infty} \frac{\sup_{0<s\le t} d(X_s,x)}{\phi^{-1}(t/\log\log t)}=c,~~~\qquad\,
\Pp^x\mbox{-a.e. } \omega,~~\forall x\in M.$$\end{theorem}

\begin{proof}
Since the proof is the same as that of Theorem \ref{inf1} with some modifications, we just highlight a few differences.
Note that, by Proposition \ref{con}, the process $X$ is conservative. With the notions in the argument above, we
define the sequences $a_k$, $\sigma_k$ and sets $G_k$, $D_k$
as
$\phi(a_k)= e^{k^2}$, $\sigma_k=\sum_{i=1}^{k-1} u_i$ and
$$G_k=\big\{\sup_{\sigma_k\le s\le \sigma_{k+1}} d(X_s,X_{\sigma_k})>a_k\big\},\quad D_k=\big\{\sup_{0<s\le \sigma_{k+1}} d(X_s,x)>2a_k\big\},$$ respectively.
To conclude the proof, we
use
Theorem \ref{t:01tail} instead of
Blumenthal 0-1 law.
\end{proof}

\begin{remark}
It can be easily observed that the behavior of $\limsup$ does not change if
we consider $\sup_{0<s\le t} d(X_s,x)$ instead of $d(X_t,x)$.
However, the $\liminf$ behavior for $d(X_t,x)$ can be different from that of $\sup_{0<s\le t} d(X_s,x)$. For instance, if the process $X$ is recurrent, i.e.\ $\int_1^\infty \frac{1}{V(\phi^{-1}(t))}\,dt=\infty$, then for
all $x\in M  \setminus \mathscr{N}$, $\liminf_{t\to\infty}d(X_t,x)=0$.
\end{remark}

\section{Laws of the Iterated Logarithm for Local Times}\label{Sect4}
In this section, we discuss the LILs for local time. We assume Assumptions \ref{assmp1}, \ref{assmp11} and \ref{assmp2} throughout the section.
 Recall that, under Assumptions \ref{assmp1}, \ref{assmp11} and \ref{assmp2},   \eqref{eqn:univd*} holds for $V$ by Proposition \ref{space}, and \eqref{assum-2} is satisfied for $\phi$ by the remark below Assumption \ref{assmp2}.
Note that \eqref{eqn:univd*}  and \eqref{assum-2} are equivalent to the existence of constants
$ c_5,\cdots, c_8>1$ and $L_0>1$
such that for every $r>0$,
$$   c_5\phi (r) \leq \phi (L_0r)
\leq c_6 \, \phi (r)\quad \hbox{and}
\quad  c_7 V (r) \leq V (L_0r) \leq c_8 \,V (r).
$$
In particular,
  \begin{equation}\label{assum-4}\int_r^\infty \frac{dV(s)}{V(s)\phi(s)}\asymp \frac{1}{\phi(r)},\quad r>0.\end{equation}

\subsection{Estimates for resolvent densities}
For $\lambda>0$, we define the $\lambda$-resolvent density (i.e. the density function of the $\lambda$-resolvent operator) by
\begin{equation*}\label{e:resol}
u^\lambda(x,y)=\int_0^\infty e^{-\lambda t} p(t,x,y)\,dt.
\end{equation*}

For each $A\subset M$, set
\[
\tau_A:=\inf\{t> 0: X_t\notin A\},\quad \sigma_A:=\inf\{t> 0: X_t\in A\}
\] and $$\sigma^0_A:=\inf\{t\ge 0: X_t\in A\}.$$
For simplicity, we write $\sigma^0_x:=\sigma^0_{\{x\}}$.

For an
open subset $A\subset M$ with $A\ne M$,  define
\[
u_A(x,y)=\int_0^\infty p^A(t,x,y)\,dt,\qquad~~\, x,y\in
A,
\]
where $p^A(t,\cdot,\cdot)$ is
the
Dirichlet heat kernel of the process $X$ killed on exiting $A$, see \eqref{e:hkos1}.

\begin{proposition}\label{geenestdf}Suppose that \begin{equation}\label{assum-3}\int_0^\infty e^{-\lambda t} \frac{1}{V(\phi^{-1}(t))}\,dt \asymp \frac{\lambda^{-1}}{ V(\phi^{-1}(\lambda^{-1}))},\quad \lambda>0.\end{equation}
Then the following three statements hold.
\begin{itemize}
\item[(i)] There exist $c_1,c_2>0$ such that
\[
c_1 \frac{\phi(r)}{V(r)}\le u_{B(x,r)}(x,x)\le c_2  \frac{\phi(r)}{V(r)}~~~\,\qquad~~\mbox{ for all }~~x\in M,~~r>0.
\]
\item[(ii)] There exists $c_3>0$ such that for any $x_0\in M$, $R>0$ and
any $x,y\in B(x_0, R/4)$,
\[
\Pp^x(\sigma^0_y>\tau_{B(x_0,R)})\le c_3 \frac{\phi(d(x,y))}{V(d(x,y))}  \frac{1}{u_{B(x_0,R)}(y,y)}.
\]
\item[(iii)] It holds that $$1-\Ee^y[e^{- \sigma^0_x}]\le c_4\frac{\phi(d(x,y))}{V(d(x,y))}$$
for all $x,y\in M$.
 \end{itemize}
\end{proposition}
\begin{remark} The exponent on the right hand side of (iii) (which is $\beta-\alpha$ when $d_1=d_2=\alpha$ and $d_3=d_4=\beta$
in \eqref{eqn:univd*}
and \eqref{assum-2}) is sharp in general, and we do need this exponent later. We may be able to obtain the H\"older continuity by using the Harnack inequality in Proposition
\ref{thm:PHI_hw}, but we cannot get the sharp exponent with that approach (cf.\ Proposition \ref{p:Holder_estimates}). Another possible approach is to use the properties of the so-called resistance form (see for example, \cite{Kig12}), but they require various preparations, so we take this \lq\lq
bare-hands\rq\rq\, approach.

\end{remark}
\begin{proof}[Proof of Proposition $\ref{geenestdf}$] The following arguments are based on \cite[Section 4]{baba3} and \cite[Section 5]{bp}, but with highly non-trivial modifications due to the generality and the effects of jumps.

(i) The lower bound is easy. Set $A=B(x, r)$. By \eqref{exit} and \eqref{assum-2},
there exists a constant $c_1>0$ such that for all $x\in M$ and $r>0$,
$$\Pp^x\left(\tau_A\le c_1 \phi(r)\right)\le \frac{1}{2}$$ and so,
by conservativeness of the process (Proposition \ref{con}), we have
$$\Ee^x(\tau_A)\ge c_1 \phi(r)\Pp^x\left(\tau_A\ge c_1 \phi(r)\right)\ge \frac{c_1}{2}\phi(r).$$ We then have
\[
\frac{c_1 }{2}\phi(r)\le \Ee^x(\tau_A)=\int_Au_A(x,y)\,\mu(dy)\le u_A(x,x)\mu(A)\le c_2 V(r) u_A(x,x),
\]
where we used the fact $u_A(x,y)=u_A(y,x)
=\Pp^y (\sigma^0_x <\sigma^0_{A^c}) u_A(x,x) \le u_A(x,x)$. Thus, the lower bound is established.

Next, we prove the upper bound.
Let ${\rm{\sc Exp}}_\lambda$
be an independent exponential distributed random variable with mean $\lambda^{-1}$.
In the following, with some abuse of notation, we also use $\Pp^x$  for
the product probability of $\Pp^x$ and the law of ${\rm{\sc Exp}}_\lambda$.
We claim that there exists a constant
$c_3>0$ such that
\begin{equation}\label{eq:rlamtau}
\Pp^x({\rm{\sc Exp}}_\lambda\le \tau_A)\le (c_3 \lambda \phi(r))\wedge 1,\quad r,\lambda>0, x\in M.
\end{equation}
To prove this, we first note that
\begin{equation}\label{eq:rlamwfq}
\Pp^x(\tau_A\ge t)\le \exp (-t/ (c_3\phi(r))), \quad r,t>0, x\in M.
\end{equation}
Indeed, since for any $x\in M$ and $t,r>0$,  $$\Pp^x(\tau_{B(x,2r)}\ge t)\le \int_{B(x,2r)}p(t,x,y)\,\mu(dy)\le \frac{c_4 V(2r)}{V(\phi^{-1}(t))},$$ by
\eqref{eqn:univd*}
and \eqref{assum-2}, there is a constant $c_5>0$ such that $$\Pp^x(\tau_{B(x,2r)}\ge c_5 \phi(r))\le 1/2$$ for all $x\in M$
and $r>0$. So, by induction and the Markov property, we have for each $k\in {\mathbb N}$,
\begin{align*}
\Pp^x(\tau_A\ge c_5 (k+1)\phi(r))\le &\Ee^x\Big[1_{\{\tau_A\ge c_5 k \phi(r)\}}\Pp^{X_{c_5 k \phi(r)}}(\tau_{B(X_0,2r)}\ge c_5 \phi(r))\Big]
\le (1/2)^{k+1},
\end{align*}
which immediately yields \eqref{eq:rlamwfq}.
Using \eqref{eq:rlamwfq}, we have
\begin{align*}
\Pp^x({\rm{\sc Exp}}_\lambda\le \tau_A)=&\int_0^\infty\lambda e^{-\lambda t}\Pp^x(\tau_A\ge t)\,dt
\le \int_0^\infty\lambda e^{-\lambda t}\exp (-t/(c_3\phi(r)))\,dt\\
=&\lambda(\lambda+1/(c_3\phi(r)))^{-1}\le c_3\lambda \phi(r),
\end{align*}
so \eqref{eq:rlamtau} is established.

Now using \eqref{eq:rlamtau} with the choice of $\lambda=(2c_3\phi(r))^{-1}$, the fact that $u_A(y,x)\le u_A(x,x)$ and the strong Markov property, we have
\[
u_A(x,x)\le u^\lambda (x,x)+\Pp^x({\rm{\sc Exp}}_\lambda\le \tau_A)u_A(x,x)
\le u^\lambda (x,x)+(1/2) u_A(x,x).
\]
This, along with \eqref{a-two-sidedup}, \eqref{assum-3} and \eqref{assum-2}, gives us
\[
u_A(x,x)\le 2u^\lambda(x,x)\le 2\int_0^\infty e^{-\lambda t} \frac{1}{V(\phi^{-1}(t))}\,dt\le c_6  \frac{\phi(r)}{V(r)}.
\]

(ii) Write $A=B(x_0,R)$ and $B=B(y, c_*d(x,y))$, where $0<c_*<1$ is chosen later. Using the strong Markov property and Proposition \ref{con},
\[
u_A(y,y)=u_B(y,y)+\Ee^y\left(1-f_y(X_{\tau_B})\right)u_A(y,y),
\]
where $f_y(x):=\Pp^{x}(\sigma^0_y>\tau_A)$. Thus,
\begin{equation}\label{eq:uaubfoe}
u_B(y,y)=u_A(y,y)\Ee^y[f_y(X_{\tau_B})].
\end{equation}
Since $f_y(\cdot)$ is harmonic on $A\setminus \{y\}$,
by Proposition \ref{thm:PHI_hw} (we only use the elliptic Harnack inequality here),
there exist two constants $c_1,c_2>0$ such that
\begin{equation}\label{eq:Harconseq}
c_1\le f_y(z)/f_y(z')\le c_2,~~\quad\, \forall z,z'\in B(y, c_*kd(x,y))\setminus B,
\end{equation}
where we choose $k>0$ to satisfy $1<c_*k<3/2$. Note that $1<c_*k$ is required in order to
guarantee that $x\in B(y, c_*kd(x,y))\setminus B$.
Using the jump kernel of the process $X$ (see Proposition \ref{jump}) and the L\'evy system formula (see for example \cite[Appendix A]{CK1}), we have
\begin{align*}
\Pp^y (X_{\tau_B\wedge t}\notin B(y, c_*kd(x,y)))=&\Ee^y\Big[\int_0^{\tau_B\wedge t}\int_{B(y, c_*kd(x,y))^c}
J(X_s,u)\,\mu(du)\,ds\Big]\\
\le &\Ee^y\Big[\int_0^{\tau_B\wedge t}\int_{B(y, c_*kd(x,y))^c}
\frac {c_3\,\mu(du)\,ds}{V(d(X_s,u))\phi(d(X_s,u))}\Big]\\
\le  &\frac{c_4 \Ee^y[\tau_B\wedge t]}{\phi(c_*(k-1) d(x,y))}\le c_5(k-1)^{-d_3},
\end{align*} where in the last line we have used \eqref{assum-2}, \eqref{assum-4} and the fact that for any $x,y\in M$, $\Ee^y(\tau_B)\le c_0\phi(c_*d(x,y))$ due to
\eqref{eq:rlamwfq} (e.g.\ see \eqref{upperexit}). Note that the constant
$c_5>0$ is independent of $c_*$ and $k$.
We choose $k$ large enough and $c_*$ small enough such that $c_5(k-1)^{-d_3}<1/2$ and $1<c_*k<3/2$.
Taking $t\to \infty$ in the inequality above, we have $$\Pp^y (X_{\tau_B}\notin B(y, c_*kd(x,y)))\le 1/2.$$
Using this, \eqref{eq:uaubfoe} and \eqref{eq:Harconseq}, we find that
\begin{align*}
\Pp^{x}(\sigma^0_y>\tau_A)/2=f_y(x)/2\le &c_2\Ee^y[1_{\{X_{\tau_B}\in B(y, c_*kd(x,y))\}}f_{y}(X_{\tau_B})]
\le c_2\Ee^y[f_{y}(X_{\tau_B})]\\
= & c_2\frac{u_B(y,y)}{u_A(y,y)}\le c_6 \frac{1}{u_A(y,y)}\frac{\phi(d(x,y))}{V(d(x,y))},
\end{align*}
where we use (i) in the last inequality. We thus obtain (ii).

(iii) From \eqref{assum-3}, we know that $$ c^{-1} \frac{\lambda^{-1}}{ V(\phi^{-1}(\lambda^{-1}))}\le \int_0^\infty e^{-\lambda t} \frac{1}{V(\phi^{-1}(t))}\,dt\le c\frac{\lambda^{-1}}{ V(\phi^{-1}(\lambda^{-1}))}$$ for some constant $c\ge 1$ and $\lambda>0$. Then, for all $r>0$,
$$ c^{-1} \frac{\phi(r)}{ V(r)}\le \int_0^\infty e^{-t/\phi(r)} \frac{1}{V(\phi^{-1}(t))}\,dt\le c\frac{\phi(r)}{ V(r)},$$ which implies
that for any $s,t>0$,
\begin{equation}\label{eq:nobo3c}
\frac{\phi(s)}{V(s)}\le c^2 \frac{\phi(s+t)}{V(s+t)}.
\end{equation}
Using \eqref{eq:nobo3c},
the desired inequality is trivial when $d(x,y)\ge e^{-1}$ by taking $c_4=\frac{c^2 V(e^{-1})}{\phi(e^{-1})}$.
Let
$n\in {\mathbb N}$ be such that $e^{-n-1}\le d(x,y)<e^{-n}$ and set $\tau_m=\tau_{B(y, e^{-m})}$
for each $m\in {\mathbb N}$. Then,
\begin{align*}
&1-\Ee^y[e^{- \sigma^0_x}]=\Pp^y(\sigma^0_x\ge {\rm{\sc Exp}}_1)\\
&\le\Pp^y(\sigma^0_x\ge {\rm{\sc Exp}}_1, {\rm{\sc Exp}}_1<\tau_n)+\sum_{m=1}^n
\Pp^y (\sigma^0_x\ge {\rm{\sc Exp}}_1, \tau_m\le {\rm{\sc Exp}}_1<\tau_{m-1})\\
&\quad +
\Pp^y(\sigma^0_x\ge {\rm{\sc Exp}}_1, {\rm{\sc Exp}}_1\ge \tau_0)\\
&\le \Pp^y({\rm{\sc Exp}}_1<\tau_n)+\sum_{m=1}^n
\Pp^y (\sigma^0_x\ge {\rm{\sc Exp}}_1, \tau_m\le {\rm{\sc Exp}}_1<\tau_{m-1})+
\Pp^y(\sigma^0_x\ge\tau_0)\\
&\le \Pp^y({\rm{\sc Exp}}_1<\tau_n)+\sum_{m=1}^n
\Pp^y (1_{\{\sigma^0_x\ge \tau_m, {\rm{\sc Exp}}_1\ge \tau_m, X_{\tau_m}\in B(y,e^{-m+1})\}}
\Pp^{X_{\tau_m}}({\rm{\sc Exp}}_1<\tau_{m-1}))\\
&\quad +
\Pp^y(\sigma^0_x\ge\tau_0)\\
&\le \Pp^y({\rm{\sc Exp}}_1<\tau_n)+\sum_{m=1}^n
\Pp^y (\sigma^0_x\ge \tau_m)\sup_{z\in B(y,e^{-m+1})}\Pp^{z}({\rm{\sc Exp}}_1<\tau_{B(y, e^{-m+1})})\\
&\quad + \Pp^y(\sigma^0_x\ge\tau_0)\\
&\le  c_1\phi(e^{-n})+c_2\sum_{m=1}^n \phi(e^{-n})V(e^{-m})/V(e^{-n})+ c_3\phi(e^{-n})/V(e^{-n})\\
&\le c_4\phi(e^{-n})/V(e^{-n})\le c_5\phi(d(x,y))/V(d(x,y)),
\end{align*}
where we used (i), (ii), \eqref{eq:rlamtau},
\eqref{eqn:univd*} and \eqref{assum-2} in
the fifth inequality, and
\eqref{eqn:univd*} and \eqref{assum-2} in the last line.
\end{proof}

\subsection{Existence and estimates for local times}
Let $(A_t)_{t\ge0}$ be a continuous additive functional of the process $X$, i.e.\
\begin{itemize}
\item $t\mapsto A_t$ is almost surely continuous and nondecreasing with $A_0=0$;
\item $A_t\in \mathscr{F}_t$;
\item $A_{t+s}(\omega)=A_t(\omega)+A_s(\theta_t\omega)$ for all $s,t\ge 0.$
\end{itemize} Set $T_A=\inf\{t>0: A_t>0\}$. $A_t$ is called a local time of the process $X$ at $x$, if $\Pp^x(T_A=0)=1$ and $\Pp^y(T_A=0)=0$ for all $y\notin x$. The reason that $A_t$ is called a local time at $x$ for the process $X$ is that the function $t\mapsto A_t$ is the distribution function of a measure supported on the set $\{t|X_t=x\}$, see e.g.\ \cite[V.\ 3]{BulG}. The next proposition gives us a necessary and sufficient condition for the existence of a local time.

\begin{proposition}\label{existence} The process $X$ has a local time for all $x\in M$, if and only if
\begin{equation}\label{existence-1} \int_0^1 \frac{ 1}{ V(\phi^{-1}(t))}\,dt<\infty.\end{equation} Moreover, we can choose a version of
the local time at $x$, which will denote by $l(x,t)$, by requiring the following property.
\begin{itemize}
\item[(1)] The function $(\omega, t, x)\mapsto l(x,t)(\omega)$ is jointly measurable such that
the following density of occupation formula holds for all non-negative  Borel measurable function $f$,
\begin{align}\label{e:ltf}
\int_0^tf(X_s)\,ds=\int_Mf(x)l(x,t)\,\mu(dx).
\end{align}
\item[(2)] For any $x,y\in M$ and $\lambda>0$,
\begin{equation}\label{res}\Ee^x\left( \int_0^\infty e^{-\lambda t}\, d l(y,t)\right)=u^\lambda(x,y).\end{equation}
\end{itemize}

\end{proposition}
\begin{proof}
According to \cite[Theorem 3.2]{MR}, the process $X$ has a local time for all $x\in M$ if and only if
$$u^\lambda(x,x)<\infty \quad\textrm{ for all }x\in M\textrm{ and some }\lambda>0.$$
Using Assumption \ref{assmp1} and the doubling properties of $V$ and $\phi$,
$$
u^\lambda(x,x)=\int_0^\infty e^{-\lambda t} p
(t,x,x)\,dt <\infty\quad \textrm{ for all } x\in M
$$ if and only if $$ \int_0^1 e^{-\lambda t} \frac{1}{V(\phi^{-1}(t))}\,dt<\infty, $$ which in turn is equivalent to \eqref{existence-1}.

Local times are defined up to a multiplicative constant, see \cite[V.\ 3.13]{BulG}. By \cite[Theorem 1]{GK} and \cite[VI.\ 4.18]{BulG}, we can choose a version of local times satisfying the desired properties (i) and (ii), also see the remark below \cite[Theorem 3.2]{MR}.
\end{proof}

Below we suppose that the local time $l(x,t)$ is always chosen to satisfy (1) and (2) in Proposition \ref{existence}, if \eqref{existence-1} is satisfied.
Note that, \eqref{assum-3} implies \eqref{existence-1}. By the strong Markov property and \eqref{res},
\begin{align*}
u^\lambda(x,y)=&\Ee^x \int_0^\infty e^{-\lambda t}\, d l(y,t)=\Ee^x \int_{\sigma_y^0}^\infty e^{-\lambda t}\, d l(y,t)\nonumber\\
=& \Ee^x e^{-\lambda\sigma^0_y}\Ee^y\int_0^\infty e^{-\lambda t} \,d l(y,t)= \Ee^x e^{-\lambda\sigma^0_y} u^\lambda(y,y). \label{eq:diagdom}\end{align*}
So,
\begin{equation}\label{eq:nofhoo}
\Ee^x[e^{- \sigma_y^0}]= u^1(x,y)/u^1(y,y),
\end{equation}
which is continuous because of the continuity of $p(t,x,y)$, see Proposition \ref{p:Holder_estimates}.
\qed

Let $d_2$ and $d_3$ be the constants in \eqref{eqn:univd*} and
\eqref{assum-2} respectively. Throughout the remainder of this
section, we always assume the following

\begin{assumption}\label{assmp3}
$d_3>d_2$.
\end{assumption}
The functions $V$ and $\phi$ respectively characterize the
underlying space and the process in question. Assumption
\ref{assmp3} means that the walk dimension of the process is greater
than the dimension of the space, which implies that the process
could stay at every point for efficiently long time; that is, the
local time of the process exists.

The following lemma is easy.
\begin{lemma}\label{regular-1} Under Assumption $\ref{assmp3}$,
 \eqref{assum-3} holds. In particular,
\begin{equation}\label{ass-v-phi} \int_0^t\frac{1}{V(\phi^{-1}(s))}\,ds \asymp \frac{t}{V(\phi^{-1}(t))},\quad t>0,\end{equation}  and so \eqref{existence-1} is satisfied. \end{lemma}
\begin{proof}
Let
$f(t):=\frac{1}{V(\phi^{-1}(t))}$ and
$$w(\lambda):=\int_0^\infty e^{-\lambda t} f(t)\,dt=\lambda^{-1}\int_0^\infty e^{-s} f(s/\lambda)\,ds.$$
Since $f$ is decreasing,
we see that
$$
w(\lambda) \ge \lambda^{-1}\int_{1/2}^{1} e^{-s} f(s/\lambda)\,ds \ge \lambda^{-1}f(1/\lambda)\int_{1/2}^{1} e^{-s}\, ds=c_0\lambda^{-1}f(1/\lambda).$$
On the other hand, it follows from \eqref{eqn:univd*} and \eqref{assum-2} that
 \begin{equation}\label{assum-qq}
 c_1\Big(\frac{R}{r}\Big)^{d_1/d_4}\le \frac{V(\phi^{-1}(R))}{V(\phi^{-1}(r))}\le c_2\Big(\frac{R}{r}\Big)^{d_2/d_3}\end{equation}
 holds for all $0<r\le R$ and some constants $c_1,c_2>0$. This along with the assumption $d_3>d_2$ yields that
 \begin{align*}
\frac{\lambda w(\lambda)}{ f(1/\lambda)}
&= \int_0^1 e^{-s} \frac{f(s/\lambda)}{ f(1/\lambda)}\, ds +\int_1^\infty e^{-s} \frac{f(s/\lambda)}{ f(1/\lambda)}\, ds\\
&\le c_2\int_0^1 e^{-s}s^{-d_2/d_3}\, ds +\int_1^\infty e^{-s} \, ds <\infty.
\end{align*}
We have proved \eqref{assum-3}.

 We now verify \eqref{ass-v-phi}.
  By the increasing properties of $V$ and $\phi$, for any $t>0$, $$\int_0^t\frac{1}{V(\phi^{-1}(s))}\,ds \ge \frac{t}{V(\phi^{-1}(t))}.$$
The upper bound of \eqref{ass-v-phi} can be obtained
 from \eqref{assum-3} as follows:
$$\int_0^t\frac{1}{V(\phi^{-1}(s))}\,ds \le e\int_0^t e^{-s/ t} \frac{1}{V(\phi^{-1}(s))}\,ds \le e\int_0^\infty e^{-s/ t} \frac{1}{V(\phi^{-1}(s))}\,ds \le \frac{c_3t}{V(\phi^{-1}(t))}.$$ The proof is complete. \end{proof}

From now on, we will always consider versions of the local time at
$x$, denote by $l(x,t)$, satisfying the results in  Proposition
\ref{existence}. The following statement is Kac's moment formula of
the local time. Since \eqref{assum-3} implies \eqref{res}, this
directly follows from \cite[Theorem 3.10.1]{MR1}.
\begin{proposition}\label{kac}
For any $x, y_i\in M$ with $1\le i\le n$ and $t>0$,
$$\Ee^x\Pi_{i=1}^n l(y_i,t)=\sum_{\pi} \Ee^xl(y_{\pi_1}, t)\cdots \Ee^{y_{\pi_{n-1}}}l(y_{\pi_n}, t),$$ where the sum runs over all permutations $\pi$ of $\{1,\ldots, n\}.$ In particular, for any $x,y\in M$ and $n\ge 1$,
$$\Ee^x(l(y,t))^n=n!\Ee^xl(y, t) \big(\Ee^y l(y, t)\big)^{n-1}.$$  \end{proposition}

Proposition \ref{geenestdf} combining with some general theory
yields the following. (See \cite[Theorem 1.1]{Cro} for the discrete
version.)

\begin{proposition}
There exists a positive constant $c_1>0$ such that for all
$x,y,z\in M$ and $u,\delta> 0$,
\begin{equation}\label{eq:ltex3}
\Pp^z(\sup_{0\le t\le u}|l(x,t)-l(y,t)|>\delta)\le 2e^ue^{-c_1\delta\sqrt{V(d(x,y))/\phi(d(x,y))}}.
\end{equation}
\end{proposition}
\begin{proof}
Let
\begin{align}\label{e:defq}
q(x,y):= (1-\Ee^x[e^{- \sigma^0_y}]\Ee^y[e^{- \sigma^0_x}])^{1/2}.
\end{align}
Note that, since $y\mapsto \Ee^y[e^{- \sigma_x}]$
is continuous (see \eqref{eq:nofhoo}),  by \cite[V. 3.25 and 3.28]{BulG}
$$
\Pp^z(\sup_{0\le t\le u}|l(x,t)-l(y,t)|>\delta)\le
2e^ue^{-\delta/(2q(x,y))}.$$ Since Proposition \ref{geenestdf}(iii)
implies  that
\begin{align}\label{e:ineq}
q(x,y)\le
(1-\Ee^x[e^{- \sigma^0_y}])+(1-\Ee^y[e^{- \sigma_x^0}])\le c_1\phi(d(x,y))/V(d(x,y)),
\end{align}
the proof is complete.
\end{proof}

The next proposition is an analogue of \cite[Lemma 5.5]{fst}. Since
we do not have self-similarity of the process, serious modifications
of the proof are needed. We will also use  a version of Garsia's
lemma (Lemma \ref{thm:Garlw}), which is proved in Appendix
\ref{A:2}.

\begin{proposition}\label{thm:estlochold}
There exist a version of the local time $l(x,t)(\omega)$ such that
almost surely $(x,t) \to l(x,t)(\omega)$ is continuous; moreover,
there exist constants $c_1,c_2>0$ such that for all $z\in M$,
$L,u,A>0$,
\begin{align*}&\Pp^z(\sup_{d(x,y)\le L}\sup_{0\le t\le u}|l(x,t)-l(y,t)|\ge A)\\
&\le \frac{c_1V(\phi^{-1}(u) \vee L)^2}{V(L)^{2}}
\exp\left(-c_2A\frac{ V(\phi^{-1}(u)\vee L)}{\phi(\phi^{-1}(u)\vee
L)}\sqrt{\frac{V((L/\phi^{-1}(u)) \wedge 1)}{\phi((L/\phi^{-1}(u))
\wedge 1)}} \right)
\end{align*}
\end{proposition}

\begin{proof}
First note that Assumption
\ref{assmp3} and  \eqref{e:ineq} (where $q$ is defined by \eqref{e:defq}) imply that the local time $l(x,t)(\omega)$
exists and it is jointly
continuous almost surely.

In fact,
since $\sup_{z \in M} u^1(z,z)<\infty$, by \eqref{eq:nofhoo} we see that for any $x,y\in M$,
$$
d_1(x,y)^2:=u^1(x,x)+u^1(y,y)-2u^1(x,y) \le 2(\sup_{z \in M}u^1(z,z)) q(x,y)^2
\le (c_0')^2q(x,y)^2.
$$
Moreover,
by \eqref{eqn:univd*}, \eqref{assum-2},  Assumption \ref{assmp3} and  \eqref{e:ineq}, for any $x_0\in M$ and any $x,y \in B(x_0,1)$
it holds that
$$q(x,y) \le c_1' \frac{\phi(d(x,y))}{V(d(x,y))}\le c_2'd(x, y)^{d_3-d_2}.$$
Thus,  for all $x \in B(x_0,1)$ and
$\eps \in (0, c_0'c_2')$ small
enough so that $(\eps/(c_0'c_2'))^{1/(d_3-d_2)}+d(x_0,x)<1$, we have
\begin{align*}
\mu(\{y\in B(x_0,1) : d_1(x,y) <\eps\}) &\ge \mu\big(B(x,(\eps/(c_0'c_2'))^{1/(d_3-d_2)})\big)\\
& \ge c_3' V\big((\eps/(c_0'c_2'))^{1/(d_3-d_2)}\big)  \ge
c_4' \eps^{d_2/(d_3-d_2)}.
\end{align*}
Therefore,
by \cite[Theorem 6.3.3]{MR1}
(with $T=B(x_0,1)$, $d_X=d_1$ and $\mu$ being $\frac{1}{\mu(B(x_0,1))}\mu$), we have the almost sure continuity of the mean zero Gaussian process $\{G_1(x):x\in B(x_0,1)\}$ with covariance $u^1(\cdot,\cdot)$,
and so
by \cite[Theorem 9.4.1]{MR1},
$\{l(x,t): x\in B(x_0,1), t\ge 0\}$ is jointly continuous almost surely. Since this is satisfied for any $x_0\in M$, $\{l(x,t): x\in M, t\ge 0\}$ is jointly continuous almost surely.

Since we
will use a scaling argument in the remainder of the proof,
we prepare a scaled distance and a scaled measure.
Below, without loss of generality, we assume $\phi(1)=1$.
For each $\delta>0$, define a metric $d_\dlt$ and a measure $\mu_{\dlt}$ on $M$ by
\begin{equation}\label{e:dmu}\begin{split}
d_\dlt(x,y):=&\delta^{-1} d(x,y),\,~~\quad~~\forall x,y\in M,\\
\mu_{\dlt}(J):=&V(\delta)^{-1}\mu (J),\,~~\quad~~\forall J\subset {\mathcal B} (M).
\end{split}\end{equation} For $\delta>0$, let $(M,d_\dlt,\mu_{\dlt})$ be the scaled metric measure space defined by \eqref{e:dmu}, and
$X^{(\delta)}:=
\{ X_{\phi(\delta) t}: t \ge 0\}$ be the
scaled process in $(M,d_\dlt,\mu_{\dlt})$. We also let
$$
V_{(\delta)}(r)=V(\delta r)/V(\delta),  \quad \phi_{(\delta)}(r)=\phi(\delta r)/\phi(\delta)$$
and $$B_{d_\dlt}(x,r)=\{x \in M : d_\dlt (x,y) <r \}.$$
Then, $\mu_{\dlt}(B_{d_\dlt}(x,r)) \asymp V_{(\delta)}(r) $ uniformly on $\delta, r>0$ and $x\in M$,
\begin{align}\label{e:sn1}
  c_1 \Big(\frac Rr\Big)^{d_1} \leq \frac{V_{(\delta)}
(R)}{V_{(\delta)} (r)} \ \leq \ c_2 \Big(\frac Rr\Big)^{d_2}~\quad \hbox{ for
every }~\delta>0,  0<r<R<\infty,
\end{align}
and
\begin{align}\label{e:sn2}
c_3\Big(\frac{R}{r}\Big)^{d_3}\le \frac{\phi_{(\delta)}(R)}{\phi_{(\delta)}(r)}\le c_4\Big(\frac{R}{r}\Big)^{d_4}\quad\hbox{ for
every }~\delta>0,  0<r<R<\infty.
\end{align}
In particular, if $(M,d,\mu)$ is an $\alpha$-set, i.e.\ satisfies \eqref{eq:d-set},
then it is easy to see that
$(M,d_\dlt,\mu_{\dlt})$ with $V(r)=r^\alpha$
 is also an $\alpha$-set, and $\mu_{\dlt}$
satisfies \eqref{eq:d-set} with the same constants $c_1,c_2>0$.

Note that the transition density function $p^{(\delta)}(t, x, y)$ of $X^{(\delta)}$ with respect to the
measure $\mu_{(\delta)}$
is related to that of $X$ by the formula
$$p^{(\delta)}(t, x, y)= V(\delta)p(\phi(\delta)t, x, y)$$ for all $t>0$ and $x,y\in M$.
Thus, from Assumptions \ref{assmp1}  we have that
all $x, y\in M$ and $t, \delta \in (0,\infty)$,
\begin{align*}
&p^{(\delta)}(t,x,y)\le C_1\bigg( \frac{1}{V_{(\delta)}(\phi_{(\delta)}^{-1}(t))} \wedge \frac{t}{V_{(\delta)}(d_\dlt(x,y))\phi_{(\delta)}(d_\dlt(x,y))}  \bigg),
\\
&C_2\bigg( \frac{1}{V_{(\delta)}(\phi_{(\delta)}^{-1}(t))} \wedge \frac{t}{V_{(\delta)}(d_\dlt(x,y))\phi_{(\delta)}(d_\dlt(x,y))}  \bigg)
\le p^{(\delta)}(t,x,y).
\end{align*}
Let
$l^{(\delta)}(x,t)$ be its local time
with respect to the measure $\mu_{(\delta)}$, which exists by Proposition \ref{existence}, \eqref{e:sn1}, \eqref{e:sn2} and the assumption $d_2<d_3$. Let $\Pp^\cdot_{(\delta)}$ be its probability space.

In the following, set $\delta'=\delta^{-1}$.
Then, from \eqref{e:ltf}
we see that  $ (V(\delta') /\phi(\delta')) l(y,\phi(\delta') t)$ under
$\Pp^{x}$ corresponds to $l^{(\delta')}(y,t)$
under $\Pp^x_{(\delta')}$.
Thus, choosing $\delta=(1/\phi^{-1}(u))\wedge L^{-1}$,
we have
\begin{equation}\label{eq:hoebc}\begin{split}&\Pp^z\Big(\sup_{d(x,y)\le L}\sup_{0\le t\le u}|l(x,t)-l(y,t)|\ge A\Big)\\
&=
\Pp^{z}\Big(
\sup_{d(x,y)\le L}\sup_{0\le t\le u/\phi(\delta')}V(\delta') /\phi(\delta')\\
&\qquad\qquad\qquad~~ ~~~|l(x,\phi(\delta')t)
-l(y,\phi(\delta')t)|\ge AV(\delta') /\phi(\delta')\Big)\\
&\le
\Pp_{\dltp}^{z}\Big(
\sup_{d_\dltp(x, y)\le \delta L}\sup_{0\le t\le u/\phi(\delta')}|l^{\dltp}(x,t)
-l^{\dltp}(y,t)|\ge AV(\delta') /\phi(\delta')\Big)\\
&\le
\Pp_{\dltp}^{z}\Big(
\sup_{d_\dltp(x, y)\le \delta L}\sup_{0\le t\le 1}|l^{\dltp}(x,t)
-l^{\dltp}(y,t)|\ge AV(\delta')/ \phi(\delta')\Big).
\end{split} \end{equation}
Set $U(r)=\sqrt{\phi(r)/V(r)}$ and $H=B_{d_\dltp}(x_0,1/2)$ for some $x_0\in M$, and define
{\small
\begin{align*}
\Gamma_{\delta'}(H):=\iint_{H\times H}\Big(\exp\Big(c_*\frac{\sup_{0\le t\le 1}|l^{\dltp}(x,t)
-l^{\dltp}(y,t)|}{U(d_\dltp(x,y))}\Big)-1\Big)\,\mu_\dltp(dx)\,\mu_\dltp(dy),\\
F_{\delta'}:=\iint_{d_\dltp(x,y)\le 1}\Big(\exp\Big(c_*\frac{\sup_{0\le t\le 1}|l^{\dltp}(x,t)
-l^{\dltp}(y,t)|}{U(d_\dltp(x,y))}\Big)-1\Big)\,\mu_\dltp(dx)\,\mu_\dltp(dy),
\end{align*}}
for small constant $c_*>0$.
Clearly $\Gamma_{\delta'}(H)\le F_{\delta'}$ and, by \eqref{eqn:univd*} and \eqref{assum-2},
 \begin{equation}\label{assum-pp}c_{L}
 \Big(\frac{R}{r}\Big)^{(d_3-d_2)/2}\le \frac{U(R)}{U(r)}\le c_{U}
 \Big(\frac{R}{r}\Big)^{(d_4-d_1)/2}\end{equation} holds for all $0<r\le R$ and some positive constants $c_{L}, c_{U}$.
We will prove in the end of this proof that
$\Ee_{\dltp}^{z}[F_{\delta'}]$ is uniformly bounded (with respect to $\delta$) so that $\Gamma_{\delta'}(H)\le F_{\delta'}<\infty$.
Assuming this fact for the moment, we can apply Lemma \ref{thm:Garlw} with
$\Psi(x)=e^{c_*x}-1$ and $q(u)=U(u)$, and deduce
\begin{eqnarray*}
|l^{\dltp}(x,t)-l^{\dltp}(y,t)|\le c_0\int_0^{d_\dltp(x,y)}\log(c_1\Gamma_{\delta'} (H)V_{(\delta')}(u)^{-2}+1)\,
\frac{U(u)du}{u}
\end{eqnarray*}
for  $\mu_{(\delta)}$-almost all $x,y\in B_{d_\dltp}(x_0,1/16)$ and $t\le 1$,
and $c_0,c_1$ are independent of $x_0$.
Due to \eqref{e:sn1} and \eqref{assum-pp}, as stated in Lemma \ref{thm:Garlw}
  the above estimate holds for $l^{\dltp}(y,t)$ under $\Pp^{z}_{\dltp}$ uniformly (i.e.
with the same constants $c_0,c_1>0$ for all $\delta>0$).
By \eqref{assum-pp} again, there exist constants $c_2, c_3>0$ independent of $\delta$ such that for
$\mu_{(\delta)}$-almost all  $x,y\in M$
with $d_\dltp(x,y)\le \delta L$ and $t\le 1$,
\begin{equation}\label{e:lholcon}\begin{split}
|l^{\dltp}(x,t)-l^{\dltp}(y,t)|\le &c_0\int_0^{\delta L}\log(c_1F_{\delta'} V_{(\delta')}(u)^{-2}+1)\,
\frac{U(u)du}{u}
\\
\le& c_2 U(\delta L)\big(\log(1+c_3F_{\delta'} V_{(\delta')}(\delta L)^{-2})\big).\end{split}
\end{equation}
Indeed,
by \eqref{e:sn1} and  \eqref{assum-pp},
\begin{align*}
&\int_0^{\delta L}\log(c_1F_{\delta'} V_{(\delta')}(u)^{-2}+1)\,
\frac{U(u)du}{u}
\\
&\le \sum_{k=0}^\infty \big(\log(1+c_1F_{\delta'} V_{(\delta')}(\delta L/2^{k+1})^{-2})\big)
U(\delta L/2^{k})\\
&\le  c_2'\big(\log(1+c_3F_{\delta'} V_{(\delta')}(\delta L)^{-2})\big)U(\delta L)
\sum_{k=0}^\infty 2^{-k(d_3-d_2)/2}\\
&\le c_2 U(\delta L)\big(\log(1+c_3F_{\delta'} V_{(\delta')}(\delta L)^{-2})\big).
\end{align*}
Plugging this into \eqref{eq:hoebc}, we have
\begin{align*}
&\Pp^z\Big(\sup_{d(x,y)\le L}\sup_{0\le t\le u}|l(x,t)-l(y,t)|\ge A\Big)\\
&\le\Pp_{\dltp}^{z}\Big(
c_2 U(\delta L)\log(1+c_3F_{\delta'} V_{(\delta')}(\delta L)^{-2})>AV(\delta')/ \phi(\delta')\Big)\\
&=\Pp_{\dltp}^{z}\Big(\log(1+c_3F_{\delta'} V_{(\delta')}(\delta L)^{-2})\ge c_2^{-1}AV(\delta')/ (U(\delta L)\phi(\delta'))  \Big)\\
&\le e^{-c_2^{-1}AV(\delta')/ (U(\delta L)\phi(\delta'))}\Big(1+c_3\Ee_{\dltp}^{z}[F_{\delta'}]/V_{(\delta')}(\delta L)^{2}\Big)\\
&\le\frac{c_4}{V_{(\delta')}(\delta L)^{2}}e^{-c_2^{-1}AV(\delta')/ (U(\delta L)\phi(\delta'))}\Big(1+\Ee_{\dltp}^{z}[F_{\delta'}]\Big)\\
&= \frac{c_4V(\phi^{-1}(u) \vee L)^2}{V(L)^{2}}
e^{-c_2^{-1}AV(\phi^{-1}(u)\vee L)/ (U((L/\phi^{-1}(u)) \wedge 1)
\phi(\phi^{-1}(u)\vee L))}\Big(1+\Ee_{\dltp}^{z}[F_{\delta'}]\Big) ,
\end{align*}
where we used Chebyshev's inequality in the second inequality, the fact that
$\delta L\le 1$ (so that $V_{(\delta')}(\delta L)\le 1$)
in the third inequality and put
$\delta= (1/\phi^{-1}(u)) \wedge L^{-1}$ in the last equality.

Finally, we will check the integrability of $F_{\delta'} $.
Using \eqref{eq:ltex3} for $l^{\dltp}(y,t)$ under $\Pp^{z}_{\dltp}$
(note that \eqref{eq:ltex3} holds uniformly, i.e.\ with the same constant $c_5>0$ for all $\delta'>0$), we have
\begin{equation*}
\Pp_{\dltp}^{z}\Big(\sup_{0\le t\le 1}|l^{\dltp}(x,t)
-l^{\dltp}(y,t)|\ge k U(d_\dltp(x,y))\Big)
\le 2e^{1-c_5k}.\end{equation*}
Let $c_*=c_5/2$, and
\begin{align*}
I_\dltp(x,y,s)
=\exp\Big(c_*\frac{\sup_{0\le t\le s}|l^{\dltp}(x,t)
-l^{\dltp}(y,t)|}{U(d_\dltp(x,y))}\Big).
\end{align*}
Thus,  we have
\begin{align*}
\Ee^{z}_{\dltp}&[I_\dltp(x,y,1)]\\
\le &\sum_{k=0}^\infty e^{c_*(k+1)}\Pp_{\dltp}^{z}\left(k\le \frac{\sup_{0\le t\le 1}|l^{\dltp}(x,t)
-l^{\dltp}(y,t)|}{U(d_\dltp(x,y))}\le k+1\right)\\
\le &
2e^{1+c_*}\sum_{k=0}^\infty e^{-c_5k/2}=:K<\infty.
\end{align*}
Note that this value is uniformly bounded for all $\delta'>0$.
Take an open covering
\[
\big\{d_\dltp(x,y)\le 1\big\}\subset \cup_{i} (B_\dltp(x_i, 2) \times B_\dltp(x_i, 2))
\]
such that each point in $\{d_\dltp(x,y)\le 1\}$ is covered
by at most a (uniformly) finite number
of $\{B_\dltp(x_i, 2) \times B_\dltp(x_i, 2)\}_i$, say $C_0$. Using the doubling property of the volume and
the assumption that balls are relatively compact, such a covering is possible.
For each $x,y$ with $d_\dltp(x,y)\le 1$,
{\small \begin{align*}
\Ee^{z}_{\dltp}[I_\dltp(x,y,1)-1]&=\Ee^{z}_{\dltp}\bigg[1_{\{\sigma_{B_\dltp(x_i, 2)}\le 1\}}
\Ee^{X^{\dltp}_{\sigma_{B_\dltp(x_i, 2)}}}\Big[I_\dltp(x,y,1-\sigma_{B_\dltp(x_i, 2)})-1\Big]\bigg]\\
&\le (K-1) \Pp^{z}_{\dltp}(\sigma_{B_\dltp(x_i, 2)}\le 1).
\end{align*}
}
So
\begin{align*}
\Ee_\dltp^{z}[F_{\delta'}]
&=\iint_{d_\dltp(x,y)\le 1}\Ee^{z}_{\dltp}[I_\dltp(x,y,1)-1]\,d\mu_\dltp(x)\,d\mu_\dltp(y)\\
&\le  c_6(K-1) \sum_i \Pp^{z}_{\dltp}(\sigma_{B_\dltp(x_i, 2)}\le 1).
\end{align*}
Here we note that $\mu_\dltp (B_\dltp(x_i, 2))\le c_6' V(2)$, i.e.\ $\mu_\dltp (B_\dltp(x_i, 2))$ is uniformly bounded.  Noting that
\begin{align*}
& \Ee_{\dltp}^{z}\left[\int_{B_\dltp(x_i,4)}l^\dltp(y,4)\,\mu_\dltp (dy)\right]
=
\Ee_{\dltp}^{z}\left[\int_0^41_{B_\dltp(x_i,4)}(X^\dltp_s)\,ds\right]\\
&\ge
\Ee_{\dltp}^{z}\left[\int_{\sigma_{B_\dltp(x_i, 2)}}^{3+\sigma_{B_\dltp(x_i, 2)}}1_{B_\dltp(x_i,4)}(X^\dltp_s)\,ds:\sigma_{B_\dltp(x_i, 2)}\le 1\right]\\
&\ge \Ee_{\dltp}^{z}\left[
\Ee_{\dltp}^{X^\dltp_{\sigma_{B_\dltp(x_i, 2)}}}\Big[{\int_{0}^{3}1_{B_\dltp(x_i,4)}(X^\dltp_s)\,ds\Big]
1_{\{\sigma_{B_\dltp(x_i, 2)}\le 1\}}}\right]\\
&\ge  c_7\Pp^{z}_{\dltp}(\sigma_{B_\dltp(x_i, 2)}\le 1),
\end{align*}
where the last inequality is due to the fact that
$$\Ee_{\dltp}^{X^\dltp_{\sigma_{B_\dltp(x_i, 2)}}}\left[\int_{0}^{3}1_{B_\dltp(x_i,4)}(X^\dltp_s)\,ds\right]$$
is uniformly bounded from below.
Indeed, since $\phi_{(\delta')}(1)=1$ for all $\delta'>0$,
using Proposition \ref{w-ndle} for the scaled process and the semigroup property for the Dirichlet heart kernel, we have
\begin{align*}
&\inf_{w\in B_\dltp(x_i,2)}\Ee_{\dltp}^{w}\bigg[\int_{0}^ 3 1_{B_\dltp(x_i,4)}(X^\dltp_s)\,ds\bigg]
\ge 3
\inf_{w\in B_\dltp(x_i,2)}\Pp_{\dltp}^{w}(\tau_{B_\dltp(x_i,4)}\ge 3)\\
&= 3\inf_{w\in B_\dltp(x_i,4)}\int_{B_\dltp(x_i,2)}
p^{(\delta), B_\dltp(x_i,4)}(3,w,y)
\,\mu_\dltp (dy)\ge c_8.
\end{align*}
We thus obtain
\begin{align*}
\sum_i \Pp^{z}_{\dltp}(\sigma_{B_\dltp(x_i, 2)}\le 1)&\le c_9
\sum_i\Ee_{\dltp}^{z}\left[\int_{B_\dltp(x_i,4)}l^\dltp(y,4)\,\mu_\dltp (dy)\right]\\
&\le  c_{10}
\Ee_{\dltp}^{z}\left[\int_{M^\dltp}l^\dltp(y,4)\,\mu_\dltp (dy)\right]=
4c_{10},
\end{align*}
so we conclude $\Ee_{\dltp}^{z}[F_{\delta'}]$ is uniformly bounded.
\end{proof}
\begin{remark}\label{thm:FSTmis}
In lines 8 and 12 of \cite[p.\ 526]{fst}, $(N/(1-c))^{n(t)}$ should be changed to
$N^{n(t)\rho}/(1-c)^{n(t)\rho/2}$. Because of the typos, in the statement
of \cite[Lemma 5.5]{fst}, $\exp\big(-c_{55}ta^\rho\delta^{-\rho\theta/2}\big)$ should be changed to
$\exp\big(-c_{55}t^{(1+d_s/2)\rho/2}a^\rho\delta^{-\rho\theta/2}\big)$.
\end{remark}

\subsection{
Laws of the iterated logarithm for the maximum of local times and
ranges of processes} In the subsection, we always assume that
Assumption \ref{assmp3} is satisfied. In particular, according to
Proposition \ref{thm:estlochold}, the joint continuous version of
the local time of the process $X$, which is denoted by  $l(x,t)$ as
before,  exists for all $x\in M$. Denote by
  $$L^*(t)=\sup_{x \in M} l(x,t),\quad t>0.$$
 We will establish two LILs for $L^*(t)$.
\begin{remark} Even for one-dimensional L\'{e}vy process, some mild assumptions like Assumption \ref{assmp3} on characteristic exponent (also called symbol) are required to establish LILs of associated local times, see \cite{Wee}.
\end{remark}

\ \

First, we have the following LIL for $L^*(t)$.

\begin{theorem} \label{local-1}  Under Assumption $\ref{assmp3}$,
there exists a constant $c_0\in (0,\infty)$ such that
$$\limsup_{t\to\infty} \frac{L^*(t)}{t/V(\phi^{-1}(t/\log\log t))}= c_0,~~~\quad\,
\Pp^x\mbox{-a.e. } \omega,~~\forall x\in M.$$ \end{theorem}

We need the following tail probability estimate for the local time $l(x,t)$.

\begin{lemma}\label{local-lemma-1} Under Assumption $\ref{assmp3}$, there exists a constant $c_1>0$ such that for all $x,y\in M$ and $t, b>0$,
$$\Pp^y\left(l(x,t)\ge  \frac{b t}{V(\phi^{-1}(t))} \right)\le 2e^{-c_1b}.$$ \end{lemma}
\begin{proof}
For any $\varepsilon>0$, by Assumption \ref{assmp1},
$$\Pp^y(d(X_s,x)\le \varepsilon)=\int_{B(x,\varepsilon)} p(s,y,z)\,\mu(dz)\le \frac{C_2}{V(\phi^{-1}(s))}\mu(B(x,\varepsilon)),$$ and so
$$\int_0^t  \Pp^y(d(X_s,x)\le \varepsilon)\,ds \le C_2\mu(B(x,\varepsilon)) \int_0^t\frac{1}{V(\phi^{-1}(s))}\,ds.$$ Combining this with the fact
$$l(x,t)= \lim_{\varepsilon\to0}\frac{1}{\mu(B(x,\varepsilon))}\int_0^t \I_{B(x,\varepsilon)}(X_s)\,ds, $$ we have \begin{equation}\label{one-m}\Ee^y (l(x,t))\le C_2\int_0^t\frac{1}{V(\phi^{-1}(s))}\,ds.\end{equation}
Furthermore, according to the estimate above and Proposition \ref{kac}, we find that
$$ \Ee^y \Big (l(x,t)^n\Big)\le n! \left(C_2\int_0^t\frac{1}{V(\phi^{-1}(s))}\,ds\right)^n,\quad n\ge 0,$$ which implies that
\begin{equation}\label{exp-m}\Ee^y\left(\exp\left(\frac{l(x,t)}{2C_2 \int_0^t \frac{1}{V(\phi^{-1}(s))}\,ds}\right)\right)\le 2.\end{equation} The desired assertion is a direct consequence of the inequality above, the Chebyshev inequality and \eqref{ass-v-phi}.
\end{proof}

\begin{remark}
Alternatively one can obtain the exponential integrability
\eqref{exp-m} directly from \eqref{one-m}, by applying
Khas'misnkii's lemma, e.g.\ see \cite[Lemma B.1.2]{Si}. \end{remark}

\begin{proposition}\label{p-local}
There are constants $c_1,c_2>0$ such that for $b\ge 1$,
$$
\sup_{t>0, x\in M}
\Pp^x\left(L^*(t)\ge  \frac{bt}{V(\phi^{-1}(t))}\right)\le c_1 b^{-c_2}. $$
 \end{proposition}

\begin{proof}Let $f$ be an increasing function such that $f(1)=1$ and $\lim_{r\to\infty}f(r)=\infty$.
By \eqref{exit}, the doubling property of $\phi$ and \eqref{assum-2}, we find that for any $x\in M$ and
$t>0$ and $b \ge 1$,
\begin{align*}
&\Pp^x\left(L^*(t)\ge \frac{2bt}{V(\phi^{-1}(t))} \right)\\
&\le  \Pp^x\left(\sup_{d(z,x)\le f(b) \phi^{-1}(t)}l(z,t)\ge \frac{2bt}{V(\phi^{-1}(t))} \right)
+\Pp^x v\left(\sup_{0<s\le t}d(X_s,x)\ge f(b) \phi^{-1}(t) \right)\\
&\le \Pp^x\left(\sup_{d(z,x)\le f(b) \phi^{-1}(t)} l(z,t)\ge \frac{2bt}{V(\phi^{-1}(t))}\right)
+ \frac{c_0t}{\phi (f(b)\phi^{-1}(t))}\\
&\le \Pp^x\left(\sup_{d(z,x)\le f(b) \phi^{-1}(t)} l(z,t)\ge \frac{2bt}{V(\phi^{-1}(t))}\right)+ c_1f(b)^{-d_3}
 \end{align*} for some constant $c_1>0$.

 On the one hand, by Lemma \ref{local-lemma-1}, there is a constant $c_2>0$ such that for all $x\in M$,
 $t>0$ and $b \ge1$
\begin{align*} \Pp^x&\left(\sup_{d(z,x)\le f(b) \phi^{-1}(t)}l(z,t) \ge \frac{2bt}{V(\phi^{-1}(t))} \right) \\
\le & \Pp^x\left(\sup_{d(z,x)\le f(b) \phi^{-1}(t)} |l(z,t)-l(x,t)|\ge \frac{bt}{V(\phi^{-1}(t))} \right)
+\Pp^x\left(l(x,t)\ge \frac{bt}{V(\phi^{-1}(t))}\right)\\
\le& \Pp^x\left(\sup_{d(z,x)\le f(b) \phi^{-1}(t)} |l(z,t)-l(x,t)|\ge \frac{bt}{V(\phi^{-1}(t))}\right)
+2e^{-c_2b}.
 \end{align*}
On the other hand, according to Proposition \ref{thm:estlochold}, there are constants $c_3, c_4>0$ such that for all $t>0$ and $b\ge 1$,
\begin{align*}   \Pp^x&\left(\sup_{d(z,x)\le f(b) \phi^{-1}(t)} |l(z,t)-l(x,t)|\ge  \frac{bt}{V(\phi^{-1}(t))} \right)\\
&\le c_3 \exp\left(-c_4 b\frac{t}{V(\phi^{-1}(t))} \frac{V(f(b)\phi^{-1}(t))}{\phi(f(b)\phi^{-1}(t))}\right)\\
&= c_3 \exp\left(-c_4 b\frac{V(f(b)\phi^{-1}(t))}{V(\phi^{-1}(t))} \frac{\phi(\phi^{-1}(t))}{\phi(f(b)\phi^{-1}(t))}\right)\\
& \le c_5 \exp\left(-c_6 b f(b)^{d_1} f(b)^{-d_4}\right)
= c_5 \exp\Big(-\frac{c_6b}{f(b)^{\theta}}\Big),
 \end{align*}
 where $\theta:=d_4-d_1>0$.

Combining with all the estimates above, we find that
$$\sup_{t>0}\Pp^x\left(L^*(t)\ge \frac{bt}{V(\phi^{-1}(t))}\right)\le c_7\bigg[ f(b)^{-d_3}+ e^{-c_2b}+ \exp\left[-\Big(\frac{c_6b}{f(b)^\theta}\Big)\right]\bigg].$$ The proof is finished by taking $f(r)= r^{1/(2\theta)}$ in the inequality above . \end{proof}

Now, we are ready to prove Theorem \ref{local-1}.
\begin{proof}[Proof of Theorem $\ref{local-1}$]
(i)\textbf{(Upper bound)}: According to Proposition \ref{p-local}, we find that $$
\sup_{t>0, x\in M} \Pp^x\left(L^*(t)\ge \frac{bt}{V(\phi^{-1}(t))}\right)\to 0,\quad b\to \infty.$$ Then, according to Proposition \ref{U-LIL} and the (stronger) doubling properties of $V$ and $\phi$, we know that
$$\limsup_{t\to\infty} \frac{L^*(t)}{ t/V(\phi^{-1}(t/\log\log t))}= \limsup_{t\to\infty} \frac{L^*(t)}{ \frac{(t/\log\log t)}{ V(\phi^{-1}(t/\log\log t))}( \log\log t)} \le c_0.$$

\noindent
(ii)\textbf{(Lower bound)}: Let $R(t)=\mu(X([0,t]))$ be the range of the process. By Theorem \ref{inf2}, there is a sequence $\{t_n\}$ such that $t_n\to \infty$ as $n\to\infty$, and
 $$\sup_{0\le s\le t_n} d(X_s, x)\le c_1\phi^{-1}\left(\frac{t_n}{\log\log t_n}\right).$$ Since $R(t)\le c_2V(\sup_{0\le s \le t} d(X_s,x))$,
 $$R(t_n)\le c_3V\left(\phi^{-1}\left(\frac{t_n}{\log\log t_n}\right)\right).$$ In particular,
 \begin{equation}\label{LILrnoinf}
 \liminf_{t\to\infty} \frac{R(t)}{V\left(\phi^{-1}\left({t}/{\log\log t}\right)\right)}\le c_3.
 \end{equation}
 By the fact that
 \begin{equation}\label{eq:Lrnrel}
 t=\int_{X([0,t])} l(x,t)\,\mu(dx)\le L^*(t)R(t),
 \end{equation}
 we get
 $$\limsup_{t\to\infty} \frac{L^*(t)}{ t/V\left(\phi^{-1}\left({t}/{\log\log t}\right)\right)}\ge \limsup_{t\to\infty} \frac{t}{R(t) t/V\left(\phi^{-1}\left({t}/{\log\log t}\right)\right) } \ge \frac{1}{c_3}.$$

 From those two inequalities above, we have proved the desired assertion by zero-one law for tail events (see Theorem \ref{t:01tail}).
\end{proof}

\bigskip

Next, we turn to
another LIL.
\begin{theorem}\label{LILinfL} Under Assumption $\ref{assmp3}$,
there exists a constant $c_0\in (0,\infty)$ such that
 $$\liminf_{t\to\infty} \frac{ L^*(t)}{ (t/\log\log t)/ V(\phi^{-1}(t/\log\log t))}
= c_0,~~~\quad\,
\Pp^x\mbox{-a.e. } \omega,~~\forall x\in M.$$ \end{theorem}

\begin{proof} (i)\textbf{(Lower bound)}: Let $R(t)$ be the range of the process. Then, by \eqref{exit},
\begin{align*}  \Pp^x(R(t)\ge r)\le &\Pp^x(\sup_{0\le s \le t} d(X_s, x)\ge V^{-1}(c_1r))\le \frac{c_2t}{\phi(V^{-1}(c_1r))}.
 \end{align*} According to the doubling properties of $V$ and $\phi$,
 $$\sup_{x\in M, t>0}\Pp^x(R(t)\ge b V(\phi^{-1}(t)))\to 0,\quad b\to \infty.$$
 This, along with Proposition \ref{U-LIL} and the doubling properties of $V$ and $\phi$ again, yields that
 \begin{equation}\label{eq:rnlisup}
 \limsup_{t\to\infty} \frac{ R(t)}{ V(\phi^{-1}(t/\log\log t))\log\log t}\le c_3.
 \end{equation}
Also due to
 \eqref{eq:Lrnrel}, we get that
 \begin{align*}&\liminf_{t\to\infty} \frac{ L^*(t)}{ (t/\log\log t)/ V(\phi^{-1}(t/\log\log t))}\\
 & \ge
 \liminf_{t\to\infty} \frac{ t}{R(t) (t/\log\log t)/ V(\phi^{-1}(t/\log\log t))} \ge \frac{1}{c_3}.\end{align*}

\noindent
(ii)\textbf{(Upper bound)}: Below, we turn to prove that
$$ \liminf_{t\to\infty} \frac{ L^*(t)}{ (t/\log\log t)/ V(\phi^{-1}(t/\log\log t))}\le c_4,$$ which along with the inequality above and zero-one law for tail events
(see Theorem \ref{t:01tail}) yields the required assertion.

 Let $t_k=e^{k^2}$. Then,
 \begin{align*} \liminf_{t\to\infty} &\frac{ L^*(t)}{ (t/\log\log t)/ V(\phi^{-1}(t/\log\log t))}\\
 \le &\limsup_{k\to \infty} \frac{ L^*(t_k)}{ (t_{k+1}/\log\log t_{k+1})/ V(\phi^{-1}(t_{k+1}/\log\log t_{k+1}))}\\
 &+\liminf_{k\to\infty}\sup_{x\in M} \frac{l(x, t_{k+1})-l(x,t_k)}{(t_{k+1}/\log\log t_{k+1})/ V(\phi^{-1}(t_{k+1}/\log\log t_{k+1}))}. \end{align*}
From Theorem \ref{local-1}, \eqref{assum-qq} and the assumption $d_3>d_2$,
we know that
 $$\limsup_{k\to \infty} \frac{ L^*(t_k)}{ (t_{k+1}/\log\log t_{k+1})/ V(\phi^{-1}(t_{k+1}/\log\log t_{k+1}))}=0.$$
   So, by the Markov property and the second Borel-Cantelli lemma, it suffices to prove that there is a constant $C>0$ such that for any $x\in M$,
 $$\sum_{k=1}^\infty \Pp^x\Big(\sup_{x\in M}({l(x, t_{k+1})-l(x,t_k)})< C\frac{t_{k+1}/\log\log t_{k+1}}{V(\phi^{-1}(t_{k+1}/\log\log t_{k+1}))}|\mathscr{F}_{t_{k}}\Big)=\infty.$$

 For this, we follow the proofs of \cite[Proposition 4.8]{BK} and \cite[Theorem 3.2]{Wee} but with some significant modifications.
Note that,  using Assumption \ref{assmp1}, we have that
there is a constant $c_0=c_0(d_3) \in (0,1)$
such that
for every $t>0$ and balls  $B_1$ and $B_2$ of radius $2 \phi^{-1}(t)$ with $B_1\cap B_{2}\ne \emptyset$,
 \begin{equation}\label{e:third}\begin{split}
\inf_{t>0, z\in B_1}  & \int_{B_2} p(t,z,y)\,\mu(dy) \\
 &\ge c\inf_{t>0, z\in B_1}\int_{B_2}\left( \frac{1}{V(\phi^{-1}(t))}\wedge \frac{t}{V(d(z,y))\phi(d(z,y))}\right)\,\mu(dy) \\
  &\ge c\inf_{t>0}\left( \frac{1}{V(\phi^{-1}(t))}\wedge \frac{t}{V(8 \phi^{-1}(t) )\phi(8 \phi^{-1}(t) )}\right)\,\mu(B_2)\,\ge\, c_0,
   \end{split}\end{equation}
where in the last inequality we used the doubling properties of $V$ and $\phi$.

 Let $\gamma=-4\log (c_0/2)$  and constants
 $\rho>2$ and  $c_*>0$ will be  chosen later.
Set $s=\gamma t/\log\log t$ for $t > e^2$.
 According to
 Lemma \ref{space-c}, there exists a sequence $\{A_i\}_{i=0}^\infty$ depending on $x$ and $s$ such that
 each $A_i$ is a ball of radius $2 \phi^{-1}(s)$, $\lim_{i\to\infty}d(x,A_i)=\infty$, and the following hold:
\[
x\in A_0,~\,~~~ A_i\cap A_{i+1}\ne \emptyset~~\mbox{ for all }~i\in {\mathbb N},\,~~~
A_i\cap A_j=\emptyset ~~\mbox{ for all }~ |i-j|\ge 2.
\]
For $k\ge 1$, set
  \begin{align*} E_k=\Big\{&\sup_{x\in M} (l(x,ks)-l(x,(k-1)s))\le c_* (t/\log\log t)/ V(\phi^{-1}(t/\log\log t)),\\
   &\sup_{0\le u<s}d(X_{(k-1)s+u},X_{(k-1)s})\le \rho \phi^{-1}(s), \,\, X_{ks}\in A_{2k} \Big\}.\end{align*}
   Let
   $$
   B_1:=\big\{L^*(s)\le c_*(t/\log\log t)/ V(\phi^{-1}(t/\log\log t)\big\},$$
   $$B_2:=\big\{\sup_{0<u<s} d(X_u,X_0)\le \rho \phi^{-1}(s)\big\}~\mbox{ and }~
   B_{3,k}:=\big\{ X_{s}\in A_{2k}\big\}.$$
   By the strong Markov property, for all $x\in M$,
   \begin{equation}\label{eq:set1}\begin{split} \Pp^x\left(\bigcap_{k=1} ^{n_0} E_k| \mathscr{F}_{(n_0-1)s}\right)
       &=\Big(\prod_{k=1}^{n_0-1} \I_{E_k} \Big)\Pp^{X_{(n_0-1) s}} (E_{n_0})\\
       &=\Big(\prod_{k=1}^{n_0-1} \I_{E_k} \Big)\Pp^{X_{(n_0-1) s}} (B_1\cap B_2\cap B_{3,n_0}).
       \end{split}\end{equation}

   First, let $c_1, d_1$ and $d_4$ be the constants in \eqref{assum-qq}. For $s>0$ and $c_*>0$ with $c_* c_{1} \gamma^{-1+(d_1/d_4)}\ge 1$, using Proposition \ref{p-local}, we have
   \begin{align*}\sup_{z\in M}  \Pp^{z}(B_1^c)
   &\le \sup_{z\in M} \Pp^{z}\left(L^*(s)\ge c_* c_{1} \gamma^{-1+(d_1/d_4)}s/V(\phi^{-1}(s))\right)\\
   &\le c_2\big(c_* c_{1} \gamma^{-1+(d_1/d_4)}\big)^{-c_3},\end{align*} where in the first inequality we have used \eqref{assum-qq}, and $c_2, c_{3}$  are positive constants independent of $s$ and $c_*$. Second, according to
Propositions \ref{con} and \ref{lemma-inf}, there is a constant $c_4\in(0,1)$ such that for all
   $s>0$ and  $\rho \ge1$,
   \begin{align*}\sup_{z\in M}  \Pp^{z}(B_2^c)
   \le& \sup_{z\in M} \Pp^{z}\left(\sup_{0<u<s} d(X_u,z)\ge \rho \phi^{-1}(s)\right)
   \le c_4^\rho.
   \end{align*}
 Third, by \eqref{e:third}, for any $k\ge 1$, \begin{align*}
  \inf_{z\in A_{2(k-1)}} \Pp^z(B_{3,k})
 =\inf_{z\in A_{2(k-1)}}\int_{A_{2k}} p(s,z,y)\,\mu(dy)
  \ge &c_0.
   \end{align*}

Combining with all the estimates above and the fact
$$\Pp(D_1\cap D_2\cap D_3)\ge \Pp(D_3)-\Pp(D_1^c)-\Pp(D_2^c),$$ we find that
$$\inf_{z\in A_{2(k-1)}}\Pp^z(E_k)=\inf_{z\in A_{2(k-1)}}\Pp^z(B_1\cap B_2\cap B_{3,k}) \ge c_0- c_2\big(c_* c_{1}\gamma^{-1+(d_1/d_4)}\big)^{-c_3}- c_4^\rho.$$
Now we choose $c_*$ and $\rho$ depending on  $d_1, d_4$ and $c_i$, $i=1, \dots 4$,  large enough such that
$ \inf_{z\in  A_{2(k-1)}}\Pp^x(E_k)\ge c_0/2.$
By this and \eqref{eq:set1}, we find that for all $x\in M$ and $t > e^2$,
\begin{align*}
\Pp^x\Big(\bigcap_{k=1} ^{n_0} E_k\Big)
&\ge  (c_0/2)^{n_0}\ge
(c_0/2) \Big(\log t\Big)^{-1/4},
\end{align*}
 where $n_0=[\frac{ \log\log t}{\gamma}]+1= [\frac{ \log\log t}{-4\log (c_0/2)}]+1$.
Since there is a constant $C=C(c_*,\rho)>0$ such that
  \begin{equation*}\label{eq:set}
  \bigcap_{k=1} ^{n_0} E_k\subset \Big\{L^*(t)< C (t/\log\log t)/ V(\phi^{-1}(t/\log\log t))\Big\},\end{equation*}
we get
 for all $x\in M$ and $t > e^2$,
\begin{align*}
\Pp^x\Big\{L^*(t)<C(t/\log\log t)/ V(\phi^{-1}(t/\log\log t))\Big\}\ge (c_0/2) \Big(\log t\Big)^{-1/4},
\end{align*}
Therefore,
\begin{align*} &\Pp^x\Big(\sup_{x\in M}({l(x, t_{k+1})-l(x,t_k)})< C (t_{k+1}/\log\log t_{k+1})/ V(\phi^{-1}(t_{k+1}/\log\log t_{k+1}))|\mathscr{F}_{t_{k}}\Big)\\
&\ge
\inf_{z\in M} \Pp^z\Big(L^*(t_{k+1})< C(t_{k+1}/\log\log t_{k+1})/ V(\phi^{-1}(t_{k+1}/\log\log t_{k+1}))\Big)\\
&\ge (c_0/2) (k+1)^{-1/2}, \end{align*}   whose summation on $k$  diverges. This completes the proof.
\end{proof}

As in the proofs of Theorems \ref{local-1} and \ref{LILinfL}, let $R(t)=\mu(X([0,t]))$ be the range of the process $X$. As a direct application of previous theorems, we have the following statements for the ranges.
\begin{theorem} \label{range} Under Assumption $\ref{assmp3}$,
there exist constants $c_0,c_1\in (0,\infty)$ such that
\begin{eqnarray}
~~~~~~~~\,\,\,\,~\limsup_{t\to\infty} \frac{R(t)}{ V(\phi^{-1}(t/\log\log t))\log\log t}= c_0, ~~~\quad\,
\Pp^x\mbox{-a.e. } \omega,~~\forall x\in M, \label{eq:rnoups}\\
\liminf_{t\to\infty} \frac{R(t)}{V\left(\phi^{-1}\left({t}/{\log\log t}\right)\right)}= c_1, ~~~\quad\,
\Pp^x\mbox{-a.e. } \omega,~~\forall x\in M.\label{eq:rnlow}
\end{eqnarray}
\end{theorem}
\begin{proof}
First, the upper bound of \eqref{eq:rnoups} is already obtained in \eqref{eq:rnlisup}.
The lower bound of \eqref{eq:rnoups} is a consequence of \eqref{eq:Lrnrel} and Theorem \ref{LILinfL}.
Next, the upper bound of \eqref{eq:rnlow} is already obtained in \eqref{LILrnoinf}.
The lower bound of \eqref{eq:rnlow} is a consequence of \eqref{eq:Lrnrel} and Theorem
\ref{local-1}.
Finally, the zero-one law for tail events (Theorem \ref{t:01tail}) yields the desired results.
\end{proof}

\section{Examples: Jump Processes of Mixed Types on Metric Measure Spaces}\label{Sect5}
We now give three examples. The first one is the $\beta$-stable-like processes on $\alpha$-set.
This is the case $d_1=d_2=\alpha$ and $d_3=d_4=\beta$ in
\eqref{eqn:univd*} and \eqref{assum-2},
and our results can be written simply as Theorem \ref{stable-like} in Section \ref{section1}.

The other two examples below are essentially taken from \cite[Example 2.3(1) and (2)]{CK1}. We recall the framework on the metric measure space from here. Let $(M,d,\mu)$ be a locally compact, separable and connected metric space such that there is a strictly increasing function $V$ satisfying \eqref{e:BV} and
\eqref{eqn:univd*}, i.e.
for any $x\in M$ and $r>0$,
$
\mu(B(x,r))\asymp V(r),
$
and there exist constants $c_1, c_2>0$, $d_2\geq d_1>0$
 such that
$$
~~\mbox{   }~~\mbox{   }~\mbox{   }~~~
  c_1 \Big(\frac Rr\Big)^{d_1} \leq \frac{V
(R)}{V (r)} \ \leq \ c_2 \Big(\frac Rr\Big)^{d_2}~ \hbox{ for
every }~ 0<r<R<\infty.
$$
\begin{example}\label{exfe2}\,
Assume that there exist $0<\beta_1\le \beta_2<\infty$ and a probability measure $\nu$ on $[\beta_1,\beta_2]$ such that
$$\phi(r)=\int_{\beta_1}^{\beta_2} r^\beta\,\nu(d\beta),\quad r>0.$$ Clearly, $\phi$ is a continuous strictly increasing function such that \eqref{assum-2} holds with $d_3=\beta_1$ and $d_4=\beta_2$.
Consider a regular Dirichlet form $(\mathscr{E}, \mathscr{F})$ on $L^2(M,\mu)$ such that $\mathscr{E}$ is given by \eqref{eq:DFshape} and the
L\'evy measure $n(\cdot,\cdot)$ satisfies \eqref{eq:Levyofe} with the function $\phi$ given above. Then the associated Hunt process has the transition density function $p(t,x,y)$ satisfying
Assumption \ref{assmp1} with the functions $V$ and $\phi$ given above. Furthermore, we have the following assertions.

\begin{itemize}
\item[(i)] All the statements of theorems in Section \ref{Sect3} hold for sample paths of the process $X$.

\item[(ii)] If
$d_2<\beta_1$, then the local time of the process $X$ exists, and all the theorems in Section \ref{Sect4} hold for local times and the range of the process $X$.
\end{itemize}
\end{example}

\begin{example}\label{exfe3}\,
Consider the following increasing function
$$\phi(r)=\left(\int_{\beta_1}^{\beta_2} r^{-\beta}\,\nu(d\beta)\right)^{-1},\quad r>0,$$ where $\nu$ is a probability measure on $[\beta_1,\beta_2]\subset (0,\infty)$. We can check easily that for this example \eqref{assum-2} also holds with $d_3=\beta_1$ and $d_4=\beta_2$.
Consider a regular Dirichlet form $(\mathscr{E}, \mathscr{F})$ on $L^2(M,\mu)$ such that $\mathscr{E}$ is given by \eqref{eq:DFshape} and the
L\'evy measure $n(\cdot,\cdot)$ satisfies \eqref{eq:Levyofe} with the function $\phi$ given above. Then the associated Hunt process has the transition density function $p(t,x,y)$ satisfying
Assumption \ref{assmp1} with the functions $V$ and $\phi$ given above. Furthermore, we have the same conclusions for the process $X$ as these in Example \ref{exfe2}.
\end{example}

\begin{example}\label{exfe4}\,
We give an example where $\beta$ could be strictly larger than $2$.
Assume that $(M,d,\mu)$ enjoys the following:

\begin{itemize}
\item[(i)] $\mu$ is a $\alpha$-set, namely $d_1=d_2=\alpha$.

\item[(ii)] There exists a $\mu$-symmetric conservative diffusion on $M$
which has a symmetric jointly continuous transition density
$\{q(t,x,y): t>0, x,y\in M\}$ with the following estimates for
all $t>0, x,y\in M$:
\begin{align*}
c_{1}t^{-\alpha/\beta_*}\exp \Big(-c_{2}\Big(\frac{d(x,y)^{\beta_*}}t\Big)^{\frac 1{\beta_*-1}}\Big)
&\le  q(t,x,y)\\
&\le  c_{3}t^{-\alpha/\beta_*}\exp \Big(-c_{4}\Big(\frac{d(x,y)^{\beta_*}}t\Big)^{\frac 1{\beta_*-1}}\Big),
\end{align*}
where $\beta_*\ge 2$.
\end{itemize}

It is known that various fractals including the Sierpinski gaskets
and Sierpinski carpets satisfy the conditions and for those cases,
typically $\beta_*>2$. For example, for Sierpinski gaskets,
$\beta_*=\log5/\log2$ and $\alpha=\log5/\log2$. (see \cite{Ba1,K14}
for details.)

Now, for $0<\gamma <1$, let $\{\xi_t\}_{t>0}$
be the strictly $\gamma$-stable subordinator; namely
let $\{\xi_t\}_{t>0}$ be a one dimensional non-negative L\'evy process with the generating function $\Ee[\exp (-u\xi_t)]=
\exp (-tu^{\gamma})$. Assume further that $\{\xi_t\}_{t>0}$ is independent of
the diffusion process above. Then the subordinate process of the diffusion by the $\gamma$-stable subordinator has the following
heat kernel
\[p(t,x,y)=\int_0^{\infty}q(u,x,y)\eta_t(u)\,du\qquad \mbox{for all}~~
t>0,~x,y\in M,\]
where $\{\eta_t (u):t>0, u\ge 0\}$
is the transition density of $\{\xi_t\}_{t>0}$. It is easy to check that
$p(t,x,y)$ satisfies \eqref{e:hke} with $\beta=\gamma\beta_*$, so the conclusions of Theorem
\ref{stable-like} hold (see \cite{TKjump} for details).
\end{example}

\appendix

\section{Some Proofs and Technical Lemmas}\label{Appx}
In this appendix, we give some proofs of the results in Section \ref{secn2}, and
also present some technical lemmas that are used in the paper.

\subsection{Proofs of some results in Section \ref{secn2}}\label{A1}
\begin{proof}[Proof of Proposition $\ref{con}$]
Let $\zeta$ be the lifetime of the process $X$ and $M_0=M\setminus {\mathscr{N}}$.  By \eqref{a-two-sidedlow},
we have that for any $t >0$ and every $x\in M_0$,
$$
\Pp^x (\zeta >t ) \geq  \int_{B(x,
\phi^{-1}(t))}p(t,x,y)\,\mu(dy) \geq \int_{B(x,
\phi^{-1}(t))} \frac{C_1}{V(\phi^{-1}(t)) }\,\mu(dy) \ge C_1C_*^{-1}>0.
$$
Let $u(x):=\Pp^x (\zeta =\infty)$. Then $u(x)=\lim_{t\to \infty}
\Pp^x (\zeta >t)\geq 
C_1C_*^{-1}>0$ for every $x\in M_0$. Note that $ u(X_t)={
1}_{\{\zeta >t\} } u(X_t)= \Ee^x ( { 1}_{\{\zeta = \infty\} }|
\mathscr{F}_t) $ is a bounded martingale with $\lim_{t\to \infty} u(X_t) =
{1}_{\{\zeta = \infty\} }$. Let $\{K_j; j\geq 1\}$ be an
increasing sequence of compact sets so that $\cup_{j=1}^\infty
K_j=M$ and define $\tau_j = \inf\{t\geq 0:  X_t \notin K_j\}$. Since
$X$ admits no killings inside $M$, we have $\tau_j<\zeta$ a.s.
Clearly,  $\lim_{j\to \infty} \tau_j=\zeta$. By the optional
stopping theorem,  we have for $x\in M_0$,
\begin{align*}
u(x)  & =  \lim_{j\to \infty} \Ee^x u(X_{\tau_j})= \Ee^x \big(
\lim_{j\to \infty} u(X_{\tau_j} )\big)\\
 & =   \Ee^x (\lim_{j\to
\infty} u(X_{\tau_j} ) {1}_{\{\zeta < \infty\}}+\lim_{t\to
\infty} u(X_{t})
{1}_{\{\zeta = \infty\}}\big)\\
 &\geq   
 C_1C_*^{-1} \Pp^x (\zeta <\infty) + \Pp^x (\zeta =\infty).
\end{align*}
It follows that $\Pp^x (\zeta <\infty)=0$ for every $x\in M_0$. The proof is complete. \end{proof}

\begin{proof}[Proof of Proposition $\ref{space}$] 
 Fix a point $x_0\in M$ and let $u_t(x)=p(t,x_0,x)$.
By Proposition \ref{con}, $\|u_t\|_1=1$; on the other hand, $\|u_t\|_\infty\le \frac{C_2}{V(\phi^{-1}(t))}$. Hence,
noting $V(\infty)=\infty$, we have
$$\mu(M)\ge \frac{\|u_t\|_1}{\|u_t\|_\infty}\to \infty,\quad t\to\infty,$$ that is, $\mu(M)=\infty$. Due to (1) the measure of any ball is finite, and so $M$ is not contained in any ball, which proves $\textrm{diam }(M)=\infty$.
The last assertion immediately follows from \cite[Corollary 5.3]{GH11} and the fact that $M$ is connected.
 \end{proof}

\begin{proof}[Proof of Proposition $\ref{p:Holder_estimates}$]
For simplicity, we only deal with the case that both Assumptions \ref{assmp1} and \ref{assmp2} hold true. The proof is essentially the same as that of \cite[Theorem 4.11]{CK}, and
we shall highlight a few different steps.

 For each $A\subset [0,\infty)\times M$, define
$\sigma_A=\inf\{t>0: Z_t\in A\}$ and $A_s=\{y\in M: (s,y)\in A\}.$
Let $Q(t,x,r)=[t,t+c_0\phi(r)]\times B(x,r),$
where $c_0\in(0,1)$ is the constant in \eqref{diff}.
Then, following the argument of \cite[Lemma 6.2]{CK1} and using Proposition \ref{jump} and the L\'{e}vy system for the process $X$ (see \cite[Appendix A]{CK1}), we can obtain that there is a constant $c_1>0$ such that for
all $ x\in M \setminus \mathscr{N}$,
$t,r>0$ and any compact subset $A\subset Q(t,x,r)$
\begin{equation}\label{proof-holder-1}\Pp^{(t,x)} (\sigma_A<\tau_{Q(t,x,r)})\ge c_1\frac{m\otimes \mu(A)}{V(r)\phi(r)},\end{equation} where $m\otimes \mu$ is a product measure of the Lebesgue measure $m$ on $\R_+$ and $\mu$ on $M$. Note that
unlike \cite[Lemma 6.2]{CK1}, here \eqref{proof-holder-1} is satisfied
for all $r>0$ not only $r\in(0,1]$, which is due to the fact \eqref{diff} holds for all $r>0$.

Also by the L\'{e}vy system of the process $X$, we find that there is a constant $c_2>0$ such that for
all $ x\in M \setminus \mathscr{N}$,
$t,r>0$ and $s\ge 2r$,
\begin{equation*}\begin{split}\Pp^{(t,x)}(X_{\tau_{Q(t,x,r)}}\notin B(x,s))&=\Ee^{(t,x)} \int_0^{\tau_{Q(t,x,r)}} \int_{B(x,s)^c} J(X_v,u)\,\mu(du)\, dv\\
&\le c_2\left(\int_{r>s/2} \frac{d V(r)}{V(r)\phi(r)}\right) \Ee^x \tau_{B(x,r)}.\end{split}\end{equation*}
On one hand, by the doubling properties of $V$ and $\phi$, we have
$$\int_{r>s/2} \frac{d V(r)}{V(r)\phi(r)}=\sum_{k=0}^\infty \int_{r\in(2^{k-1}s,2^k s]} \frac{d V(r)}{V(r)\phi(r)}\le \sum_{k=0}^\infty \frac{V(2^{k}s)-V(2^{k-1} s)}{V(2^{k-1} s)\phi(2^{k-1}s)}\le c_3 \frac{1}{\phi(s)}.$$ On the other hand, for
all $ x\in M \setminus \mathscr{N}$ and $r,t>0$,
by \eqref{eq:rlamwfq} (which is proved by the doubling property \eqref{doubphiw0} of $\phi$ only),
$$\Pp^x(\tau_{B(x,r)}\ge t)\le \exp(-c_4t/\phi(r)),$$ which implies that \begin{equation}\label{upperexit}\Ee^x (\tau_{B(x,r)})=\int_0^\infty \Pp^x(\tau_{B(x,r)}\ge t)\,dt\le c_5 \phi(r).\end{equation}
Therefore, there is a constant $c_6>0$ such that for
all $ x\in M \setminus \mathscr{N}$,
$t,r>0$ and $s\ge 2r$,
\begin{equation}\label{proof-holder-2}\Pp^{(t,x)}(X_{\tau_{Q(t,x,r)}}\notin B(X,s))
\le
c_6 \frac{\phi(r)}{\phi(s)}.\end{equation}

Having \eqref{proof-holder-1} and \eqref{proof-holder-2} at hand, one can follow the argument of \cite[Theorem 4.11]{CK} to get that the H\"{o}lder continuity of bounded parabolic functions (see the definition before Proposition \ref{thm:PHI_hw}), and so the desired assertion \eqref{e:Holder} for the heart kernel $p(t,x,y).$ Furthermore, \eqref{e:Holder2} is an immediately consequence of \eqref{e:Holder}.
\end{proof}

\begin{proof}[Proof of Theorem $\ref{t:01tail}$]
The proof is similar to the one of \cite[Proposition 2.3]{BK}. For
completeness, we provide the full proof here. (See \cite{bbcan} for
the original proof.) Let $\eps>0$ and $A$ be a tail event. Fix $x_0\in M$. By the
martingale convergence theorem, $\Ee^{x_0}[1_A| {\mathcal F}_t]\to
1_A$ a.s. as $t\to \infty$. Choose $t_0$ large enough so that
\begin{equation}\label{eq:8.2}
\Ee^{x_0}|\Ee^{x_0}[1_A| {\mathcal F}_{t_0}]-1_A|<\eps.
\end{equation}
Set $Y:=\Ee^{x_0}[1_A| {\mathcal F}_{t_0}]$. Then
\begin{equation}\label{eq:8.5}
 |\Pp^{x_0}(A)-\Ee^{x_0}(Y;A)|
  = |\Ee^{x_0}(1_A;A)-\Ee^{x_0}(Y;A)|<\eps.
\end{equation}
On the other hand, using \eqref{exit} and
the doubling property of $\phi$, we can take $c_1>0$ large so that
\begin{equation}\label{eq:enwq1}
\Pp^{x_0}(\sup_{s\leq t_0}d(X_s,x_0)>
c_1\phi^{-1}(t_0))<\eps.
\end{equation}
By Proposition \ref{p:Holder_estimates}, we now choose $t_1$
large so that for all $f \in L^\infty (M)$ and $x\in M$ with  $d(x,x_0)\leq
c_1\phi^{-1}(t_0)$,
\begin{equation}\label{eq:copt}
|P_{t_1}f(x)-P_{t_1}f(x_0)|<\eps\Vert f\Vert_\infty.
\end{equation}

Since $A$ is a tail event, there exists an event $C$ such that
$A=C\circ \theta_{t_0+t_1}$. Let $f(z)=\Pp^z(C)$. Then by the Markov
property at time $t_1$,
\begin{equation}\label{eq:shif}
\Ee^w(1_C\circ \theta_{t_1})=\Ee^w \Ee^{X_{t_1}}1_C=\Ee^wf(X_{t_1})=P_{t_1}f(w).
\end{equation}
Thus the Markov property at time $t_0$ and $(\ref{eq:shif})$ further give us
\begin{equation}
\Ee^{x_0}(Y;A)=\Ee^{x_0}[Y\Ee^{X_{t_0}}(1_C\circ \theta_{t_1})] =
 \Ee^{x_0}[YP_{t_1}f(X_{t_0})]
\label{eq:8.7}\end{equation}
and
\begin{equation}\label{eq:8.8}
\Pp^{x_0}(A)=\Ee^{x_0} 1_A=\Ee^{x_0}\Ee^{X_{t_0}}(1_C\circ \theta_{t_1})
 = \Ee^{x_0}[P_{t_1}f(X_{t_0})].
\end{equation}

Let  $A_{t_0}=\{d(X_{t_0},x_0)\leq c_1\phi^{-1}(t_0)\}$. Using
\eqref{eq:enwq1} and \eqref{eq:copt}, we see that
\begin{equation}\begin{split}\label{eq:8.9}
&|\Ee^{x_0}[YP_{t_1}f(X_{t_0})]-P_{t_1}f(x_0)
\Ee^{x_0}Y| \\
& \le   2\Pp_x(  A_{t_0}^c)+
\vert\Ee^{x_0}[YP_{t_1}f(X_{t_0});A_{t_0}]
-P_{t_1}f(x_0)\Ee^{x_0}[Y;A_{t_0}]\vert \\
& <   2\eps+ \vert\Ee^{x_0}[Y| P_{t_1}f(X_{t_0})-P_{t_1}f(x_0)|
;A_{t_0}] \vert \leq 3\eps.
\end{split}\end{equation}
Similarly
\begin{equation}\label{eq:8.10}
|\Ee^{x_0} P_{t_1}f(X_{t_0})-P_{t_1}f(x_0)|\leq 3\eps.
\end{equation}

Combining $(\ref{eq:8.5}), (\ref{eq:8.7}), (\ref{eq:8.8}), (\ref{eq:8.9})$ and
$(\ref{eq:8.10})$,
\begin{align*}|\Pp^{x_0}(A)-\Pp^{x_0}(A)\Ee^{x_0}Y|\leq &|\Pp^{x_0}(A)-\Ee^{x_0}(Y;A)|\\
&+| \Ee^{x_0}[YP_{t_1}f(X_{t_0})]-P_{t_1}f(x_0)
\Ee^{x_0}Y|\\
&+|P_{t_1}f(x_0)
\Ee^{x_0}Y-\Ee^{x_0} P_{t_1}f(X_{t_0})\Ee^{x_0}Y | \le  7\eps.\end{align*}
Using this and  $(\ref{eq:8.2})$, $|\Pp^{x_0}(A)-\Pp^{x_0}(A)\Pp^{x_0}(A)|\leq 8\eps.$
Since $\eps$ is arbitrary, we deduce $\Pp^{x_0}(A)=[\Pp^{x_0}(A)]^2$, and so $\Pp^{x_0}(A)$ is 0 or 1. Since
$\Pp^x(A)=\Ee^x P_{t_1}f(X_{t_0})=P_{t_0}(P_{t_1}f)(x)$
is continuous in $x$ (which is easily seen from Proposition \ref{p:Holder_estimates}) and $M$ is connected,
we further conclude that either $\Pp^x(A)$ is $0$ for all $x\in M$ or else it is $1$ for all $x \in M$. The proof is complete.
\end{proof}

\begin{proof}[Proof of Proposition $\ref{w-ndle}$]
For any $x',y'\in B(x,r/2)$ and $t>0$,
$$p(t,x',y')=p^{B(x,r)}(t,x',y')+\Ee^x\Big(p(t-\tau_{B(x,r)},X_{\tau_{B(x,r)}},y')
: \tau_{B(x,r)} < t \Big).$$ On the one hand,
$$\Ee^x\Big(p(t-\tau_{B(x,r)},X_{\tau_{B(x,r)}},y'): \tau_{B(x,r)} < t \Big)
\le \sup_{s\le t; d(y,z)\ge r/2} p(s,z,y) \le \frac{C_2t}{V(r/2)\phi(r/2)}.$$
For any $\delta\in(0,1/2)$, any $x',y'\in B(x,\frac{1}{2} \delta r)$ and $t=\phi(\delta r)$,
\begin{align*} p(t,x',y')\ge &C_1  \left( \frac{1}{V(\phi^{-1}(t))} \wedge  \frac{t}{V(d(x',y'))\phi(d(x',y'))}\right)
\ge\frac{C_1}{V(\delta r)},\end{align*} so
$$ p^{B(x,r)}(t,x',y')\ge \frac{C_1}{V(\delta r)}-   \frac{C_2}{V(r/2)}.$$ By the doubling property of $V$, we find that $$p^{B(x,r)}(t,x',y')\ge \frac{C_3}{V(r)}$$ providing that $\delta\in(0,1/2)$ is small enough.
Having this at hand, one can follow the argument of \cite[Lemma 2.3]{BBK} and use the doubling property of $\phi$ to get the first required assertion.
The second assertion of the proposition directly follows from the argument above.
\end{proof}

\begin{proof}[Proof of Proposition $\ref{lemma-inf}$]
Here we only prove the case that Assumption \ref{assmp1} and \ref{assmp2} hold. According to \eqref{exit} and the doubling property of $\phi$, for any $r>0$ and
all $x\in M$,
$$\Pp^x\Big(\sup_{0\le s \le c_0\phi(r)} d(X_s, X_0)\le 2r\Big)\le a^*_2$$ holds with some constants $c_0>0$ and $a^*_2\in(0,1)$ independent of $x$ and $r$.  Then, for any $n\ge1$ and
 $x\in M$,
by the Markov property,
 \begin{align*} \Pp^x&(\sup_{0\le s\le c_0n \phi(r)} d(X_s, x)\le r)\\
 &\le \Ee^x\Big( \I_{\{\sup_{0\le s\le c_0(n-1) \phi(r)} d(X_s, x)\le r\}}; \Pp^{X_{c_0(n-1)\phi^{-1}(r)}}(\sup_{0\le s \le c_0 \phi(r)} d(X_s, X_0)\le 2r)\Big)\\
 &\le a^*_2\Pp^x( \sup_{0\le s\le c_0(n-1) \phi(r)} d(X_s, x)\le r).\end{align*} This proves the upper bound.

 On the other hand, according to Proposition \ref{w-ndle}, there are constants $\delta_0, c_1>0$ such that for
 all $x\in M$ and any $r>0$,
 $$p^{B(x,r)}(\delta_0\phi(r) , x',y')\ge c_1 V(r)^{-1},\quad x',y'\in B(x,r/2),$$ where $p^{B(x,r)}(t,x',y')$ denotes the Dirichlet heat kernel of the process killed by exiting $B(x,r)$. Then, choosing $m= [c_0/ \delta_0]+1$,
 \begin{align*}
 \Pp^x&(\sup_{0\le s\le \delta_0 mn \phi(r)} d(X_s, x)\le r)\\
 &= \int_{B(x,r)} p^{B(x,r)}(\delta_0mn \phi(r), x,y)\,\mu(dy)\\
 &\ge\int_{B(x,r/2)}\int_{B(x,r/2)}\ldots\int_{B(x,r/2)}p^{B(x,r)}\Big(\delta_0 \phi(r), x,x_1\Big)\,\mu(dx_1)\\
 &\qquad p^{B(x,r)}\Big(\delta_0 \phi(r), x_1,x_2\Big)\mu(dx_2)\ldots
 \int_{B(x,r/2)}p^{B(x,r)}\Big(\delta_0 \phi(r), x_{mn-1},y\Big)\,\mu(dy)\\
 &\ge \Big( c_1 V(r)^{-1} \mu(B(x,r/2))\Big)^{mn}. \end{align*} Thanks to the doubling property of $V$, there exists a constant $a_1^*\in (0,1)$ such that for all $x\in M$, $r>0$ and $n\ge 1$,
$$\Pp^x(\sup_{0\le s\le \delta_0 mn \phi(r)} d(X_s, x)\le r)\ge{a_1^*}^n.$$
By the fact that
 $$ \Pp^x(\sup_{0\le s\le c_0 n \phi(r)} d(X_s, x)\le r) \ge \Pp^x(\sup_{0\le s\le \delta_0 mn \phi(r)} d(X_s, x)\le r),$$ the proof is complete.
\end{proof}

\subsection{Some technical
results}\label{A:2}
The first result is a extended version of Garsia's lemma (\cite[Lemma 1]{Gar}), see \cite[Lemma 6.1]{bp} for a version of Garsia's
lemma for a fractal.

\begin{lemma}\label{thm:Garlw}
Let $(M,d,\mu)$ satisfy \eqref{e:BV} and \eqref{eqn:univd*}. Suppose
$q:[0,\infty)\to [0,\infty)$ is a measurable function with $q(0)=0$
and that there exist constants $C_1,C_2$ and $\gamma_1,\gamma_2$
such that
\begin{align}\label{e:asum_q}
C_1 \left(\frac{r}{R}\right)^{\gamma_1}\le \frac{q(r)}{q(R)} \le C_2 \left(\frac{r}{R}\right)^{\gamma_2} \quad \text{for every } 0<r\le R< \infty.
\end{align}
Let
$\Psi:[0,\infty) \to [0,\infty)$ be a non-negative strictly increasing convex function such that
$\lim_{u\to\infty}\Psi(u)=\infty$. For any $x_0\in M$ and $R_0>0$, let $H=B(x_0, R_0)$ and $f: H \to \Rd$ be a measurable function. If
\[
\Gamma(H):=\iint_{H\times H}\Psi\Big(\frac{|f(x)-f(y)|}{q(d(x,y))}\Big)\,\mu(dx)\,\mu(dy)<\infty,
\]
then there exist $c_1, c_2>0$ that depends only on the constants  in
\eqref{eqn:univd*} and \eqref{e:asum_q} such that
\begin{equation}\label{eq:garid}
|f(x)-f(y)|\le c_1\int_0^{d(x,y)}\Psi^{-1}\Big(\frac{c_2\Gamma(H)}{V(u)^2}\Big)\, \frac{q(u)du}{u},
\end{equation}
for $\mu\times\mu$-a.e. $(x,y)\in B(x_0, R_0/8)\times B(x_0, R_0/8)$. If $f$ is continuous, then \eqref{eq:garid} holds
every $(x,y)\in B(x_0, R_0/8)\times B(x_0, R_0/8)$.
\end{lemma}

\begin{proof}
For fixed $(x,y)\in B(x_0, R_0/8)\times B(x_0, R_0/8)$ and $k\ge 0$, let $a_k:=2^{-k+1}d(x, y)$ and $B_k$' be open balls with radii
$a_k$ such that $B_{k+1} \subset B_k$ and $x,y \in B_0 \subset H$.
We denote $f_k:=\frac{1}{\mu(B_k)}\int_{B_k}fd\mu$.
For $(z,w) \in B_{k-1}$,
we have $d(z,w) \le 2 a_{k-1}$, so
by \eqref{e:asum_q},
$C_2q(2 a_{k-1}) \ge   q(d(z,w))$. Thus, since $\Psi$ is increasing,
$$
\Psi\left(\frac{|f(z)-f(w)|}{C_0q(2a_{k-1})}\right) \le
\Psi\left(\frac{|f(z)-f(w)|}{q(d(z,w))}\right),  \quad (z,w) \in B_{k-1} \times B_k.$$
Using this,
the increasing property and
the convexity of $\Psi$ and the Jensen inequality,
   \begin{equation}\label{e:Ga1}\begin{split}
&\Psi\left(\frac{|f_{k-1}-f_{k}|}{C_2q(2a_{k-1})}\right)\\
&\le
\Psi\left(\frac{1}{\mu(B_{k-1})\mu(B_k)}\int_{B_{k-1}\times B_k} \frac{|f(z)-f(w)|}{C_2q(2a_{k-1})} \mu (dw) \mu (dz)\right) \\
&\le
\frac{1}{\mu(B_{k-1})\mu(B_k)}\int_{B_{k-1}\times B_k} \Psi\left(\frac{|f(z)-f(w)|}{q(d(z,w))}\right) \mu (dw) \mu (dz)\\
&\le
\frac{\Gamma(H)}{\mu(B_{k-1})\mu(B_k)}  \le c_1 \frac{\Gamma(H)}{V(a_k)^2},
\end{split}\end{equation}
where in the last inequality we used \eqref{e:BV} and \eqref{eqn:univd*}.

On the other hand, for $k\ge 1$
   \begin{equation}\label{e:Ga2}\begin{split}
&\int_{a_{k+1}}^{a_{k}}\Psi^{-1}\Big(\frac{c_1\Gamma(H)}{V(u)^2}\Big)\, \frac{q(u)du}{u}\\
&\ge q(2a_{k-1})\Psi^{-1}\Big(\frac{c_1\Gamma(H)}{V(a_{k})^2}\Big)\ \int_{a_{k+1}}^{a_{k}} \frac{q(u)}{q(2a_{k-1})}\frac{du}{u}\\
&\ge q(2a_{k-1})\Psi^{-1}\Big(\frac{c_1\Gamma(H)}{V(a_{k})^2}\Big)\ \int_{a_{k+1}}^{a_{k}} C_1 \left(\frac{u}{2a_{k-1}}\right)^{\gamma_1}\frac{du}{u}\\
&= C_1 q(2a_{k-1})\Psi^{-1}\Big(\frac{c_1\Gamma(H)}{V(a_{k})^2}\Big)\
(2a_{k-1})^{-\gamma_1} \int_{a_{k+1}}^{a_{k}} u^{\gamma_1-1}du\\
&= c_2 q(2a_{k-1})\Psi^{-1}\Big(\frac{c_1\Gamma(H)}{V(a_{k})^2}\Big).
\end{split}\end{equation}
Thus, by \eqref{e:Ga1} and \eqref{e:Ga2},  for $k\ge 1$,
\begin{align*}
|f_{k-1}-f_{k}| \le C_0 q(2a_{k-1}) \Psi^{-1} \left(
\frac{c_1\Gamma(H)}{V(a_k)^2}  \right) \le c_3\int_{a_{k+1}}^{a_{k}}\Psi^{-1}\Big(\frac{c_1\Gamma(H)}{V(u)^2}\Big)\, \frac{q(u)du}{u}
\end{align*}
which implies
\begin{align}
\label{e:G12}
\limsup_{k \to \infty}|f_{k}-f_{0}| \le \sum_{k=1}^{\infty}|f_{k-1}-f_{k}|
\le c_2\int_{0}^{d(x,y)}\Psi^{-1}\Big(\frac{c_1\Gamma(H)}{V(u)^2}\Big)\, \frac{q(u)du}{u}.
\end{align}

Suppose that  $f$ is continuous at $x$. Then, let $B_0=B(x, a_0)$, so that $x,y \in B_0=B(x, 2d(x,y)) \subset B(x_0, R_0)$.  By considering $B_k=B(x, a_k)$ for $k \ge 1$,
we get from \eqref{e:G12} that
$$
|f(x)-f_{0}|
\le c_2\int_{0}^{d(x,y)}\Psi^{-1}\Big(\frac{c_1\Gamma(H)}{V(u)^2}\Big)\, \frac{q(u)du}{u}.
$$
Similarly,
we get from \eqref{e:G12} that, if  $f$ is continuous at $y$ then
$$
|f(y)-f_{0}|
\le c_2\int_{0}^{d(x,y)}\Psi^{-1}\Big(\frac{c_1\Gamma(H)}{V(u)^2}\Big)\, \frac{q(u)du}{u}.
$$
Thus, if $f$ is continuous at both $x$ and $y$,
$$|f(x)-f(y)| \le |f(x)-f_{0}|+|f(y)-f_{0}|
\le 2c_2\int_{0}^{d(x,y)}\Psi^{-1}\Big(\frac{c_1\Gamma(H)}{V(u)^2}\Big)\, \frac{q(u)du}{u}.
$$
The general case follows from Lebesgue differentiation theorem (e.g.\ see \cite[Theorem 1.8]{H}).
\end{proof}

The following proposition gives an upper bound for LILs. Since it can be proved by a simple
modification of the proof of \cite[Theorem 3.1]{BK}, we skip the proof.

\begin{proposition}\label{U-LIL}
Let $X$ be a strong Markov process on $(M,d,\mu)$. Suppose $(F_t)_{t\ge0}$ is a continuous adapted non-decreasing functional of $X$ satisfying the following conditions.
\begin{itemize}
\item[(1)] There exists an increasing function $\varphi$ on $\R_+$ satisfying the doubling property and such that
$$\sup_{x\in M, t>0} \Pp^x(F_t\ge b \varphi(t))\to 0\quad\textrm{ as }b\to \infty.$$
\item[(2)] $F_t-F_s\le F_{t-s}\circ\theta_s,\quad 0<s\le t.$
\end{itemize}
Then, there exists a constant $C\in(0,\infty)$ such that
$$\limsup_{t\to\infty} \frac{F_t}{\varphi\left({t}/{\log\log t}\right)\log\log t}\le C,~~~\quad\,
\Pp^x\mbox{-a.e. } \omega,~~\forall x\in M.$$
\end{proposition}

\begin{remark}
Similar to the remark after the proof of \cite[Theorem 3.1]{BK}, Proposition \ref{U-LIL} can be used to derive upper bounds for LIL of $L^*(t)=\sup_{x\in M} l(x,t)$ and the range $R(t)=\mu(X([0,t]))$ of jump processes. Note that, in our setting the continuity of $L^*(t)$ is a consequence of Proposition \ref{p-local}, the strong Markov property and the Borel-Cantelli lemma; while one can use Theorem \ref{inf1} and the fact $R(t)\le c_1V\big(\sup_{0\le s \le t} d(X_s,x)\big)$ for all $t>0$ and some constant $c_1>0$ to obtain the continuity of $R(t).$  \end{remark}

\begin{proposition}
\label{space-c} Let  $(M,d,\mu)$ be a connected metric measure space such that
diam $M=\infty$ and
the volume doubling
condition holds, i.e. there exists $c_1>0$ such that
\[
\mu(B(x,2r))\le c_1\mu (B(x,r)),\quad x\in M, r>0.
\]
Then, for each $x_0\in M$ and $R>0$, there exists a sequence $\{A_i\}_{i=0}^\infty$ such that
each $A_i$ is a ball of radius $R$, $\lim_{i\to\infty}d(x_0,A_i)=\infty$, and the following hold:
\[
x_0\in A_0,~\,~~~ A_i\cap A_{i+1}\ne \emptyset~~\mbox{ for all }~i\in {\mathbb N},\,~~~
A_i\cap A_j=\emptyset ~~\mbox{ for all }~ |i-j|\ge 2.
\]
\end{proposition}
\begin{proof}
First, by \cite[Lemma 3.1 (i)]{KSt}, there exists a constant $N_0\in {\mathbb N}$ such that
for each $R>0$, there exists an open covering $\{B(z_i,R)\}_{i=0}^\infty$ of $M$ with the property
that no point in $M$ is more than $N_0$ of the balls. We say a subset $\Lambda$ of
$\{z_i\}_i$ is linked if for each $z_i,z_j\in \Lambda$, there is a chain $z^0=z_i, z^1,\cdots, z^l=z_j\in \Lambda$ such that $z^k\sim z^{k+1}$ (by which we mean $B(z^k,R)\cap B(z^{k+1},R)\ne \emptyset$) for all $k=0,1,\cdots, l-1$.
Take $x_0\in M$. We may assume without loss of generality that
$x_0=z_0$. For each $k\in {\mathbb N}$, we may take a linked set $G_k\subset\{z_i\}_i\cap B(x_0, 4kR)^c$
such that $\sharp G_k=\infty$. (Indeed, if there is no such linked sets, then because diam $M=\infty$ and $M$ is connected, there are infinite number of mutually disjoint and non-empty linked sets $\{L_j\}$
such that
$\sharp L_j<\infty$ and
$L_j\subset\{z_i\}_i\cap B(x_0, 4kR)^c$. We may assume that each $L_j$ is maximal
(i.e. no elements in $\{z_i\}_i\cap B(x_0, 4kR)^c\cap L_j^c$ is linked to $L_j$). Because
$M$ is connected, from each $L_j$, there exists $\hat x_j\in L_j$ such that
$B(\hat x_j,R)\cap B(x_0, 4kR)\ne \emptyset$. By construction, $\{B(\hat x_j,R)\}_j$ are mutually disjoint, but this contradicts to the volume doubling assumption.)
We fix one such a linked set $G_k$ which is maximal; we may choose $G_k\supset G_{k+1}\supset \cdots$.
Set $G_0=\{z_i\}_i$.

We now construct a desired chain inductively that contains a sequence $\{z_{m_k}\}_{k=0}^\infty\subset \{z_i\}$. Take $z_{m_0}=x_0$. For each $k\ge 0$, given $z_{m_k}\in G_k\cap B(x_0, (4k+2)R)$,
take a chain $y^k_0=z_{m_k},y^k_1,\cdots, y^k_{s_k}$ such that $y^k_i\sim y^k_{i+1}$ for $i=0,\cdots s_k-1$,
$y^k_j\in G_k\setminus G_{k+1}$, $j=0,\cdots s_k-1$ and $y^k_{s_k}=:z_{m_{k+1}}\in G_{k+1}$. Then it holds
that $z_{m_{k+1}}\in B(x_0, (4(k+1)+2)R)$. Now let $\tilde y^k_0=y^k_0$
and define $\tilde y^k_i$, $i\ge 1$ inductively as the maximum $j$ such that $y^k_j\sim
\tilde y^k_{i-1}$.
Then we have a sequence $\tilde y^k_0=z_{m_k}, \tilde y^k_1, \cdots, \tilde y^k_{s_k'}=z_{m_{k+1}}$ such that
$\tilde y^k_i\sim \tilde y^k_{i+1}$ and $\tilde y^k_i\not\sim\tilde y^k_j$ if $|i-j|\ge 2$.
By doing this procedure iteratively, and doing the same procedure (i.e. procedure to produce
$\{\tilde y^k_i\}$ from $\{y^k_i\}$) again
for each adjacent sequences (this is necessary because the sequences of balls made
by the adjacent sequences $\{\tilde y^k_0=z_{m_k}, \tilde y^k_1, \cdots, \tilde y^k_{s_k'}=z_{m_{k+1}}\}$
and $\{\tilde y^{k+1}_0=z_{m_{k+1}}, \tilde y^{k+1}_1, \cdots, \tilde y^{k+1}_{s'_{k+1}}=z_{m_{k+2}}\}$ could overlap many times),
we have the desired chain.
\end{proof}

\bigskip

{\bf Acknowledgement.}
Our first proof of Proposition \ref{thm:estlochold} was under assumption of some scaling property on the space.
We thank D.\ Croydon, C.\ Nakamura and Y.\ Shiozawa  for useful comments, 
and we also thank Professor Zhen-Qing Chen for the proof of Proposition \ref{con}. The authors are also indebted to two referees for their helpful comments and careful corrections.

\vskip 0.1truein

{\footnotesize {\bf Panki Kim}

Department of Mathematical Sciences and Research Institute of Mathematics,

Seoul National University,
Building 27, 1 Gwanak-ro, Gwanak-gu,
Seoul 08826,
Republic of Korea

E-mail: pkim@snu.ac.kr

\bigskip

{\bf Takashi Kumagai}

Research Institute for Mathematical Sciences,
Kyoto University, Kyoto 606-8502, Japan

E-mail: kumagai@kurims.kyoto-u.ac.jp

\bigskip

{\bf Jian Wang}

School of Mathematics and Computer Science, Fujian Normal University, 350007 Fuzhou,

P.R. China.

E-mail: jianwang@fjnu.edu.cn
}
\vfill

\end{document}